\documentclass[reqno,12pt]{amsart}

\usepackage{amsmath,amssymb,amsfonts,amsthm}

\newcommand{\Eqref}[1]{\eqref{#1}}
\newcommand{\Figref}[1]{Fig.~\ref{#1}}

\newcommand{\Secref}[1]{Section~\ref{#1}}
\newcommand{\Thmref}[1]{Theorem~\ref{#1}}
\newcommand{\Corref}[1]{Corollary~\ref{#1}}
\newcommand{\Lemref}[1]{Lemma~\ref{#1}}
\newcommand{\Propref}[1]{Prop.~\ref{#1}}
\newcommand{\Defref}[1]{Definition~\ref{#1}}

\newcommand{\Claimref}[1]{Claim~\ref{#1}}

\newcommand{\R}{\mathbb{R}} 
\newcommand{\Z}{\mathbb{Z}} 
\newcommand{\Prob}{\mathbb{P}} 

\newcommand{\e}{\epsilon}
\newcommand{\w}{\omega}
\newcommand{\W}{\Omega}
\newcommand{\Norm}[2]{\left\lVert{#1}\right\rVert_{#2}} 
\DeclareMathOperator*{\esssup}{ess\,sup} 
\DeclareMathOperator*{\essinf}{ess\,inf} 
\newcommand{\AND}{\textrm{ and }}

\newcommand{\almostsurely}{\textrm{a.s}}

\renewcommand{\limsup}{\varlimsup}
\renewcommand{\liminf}{\varliminf}

\DeclareMathOperator{\argmin}{argmin}

\newcommand{\E}{\mathbb{E}}
\newcommand{\Lip}{\text{Lip}} 

\makeatletter
\def\@wraptoccontribs#1#2{}
\makeatother

\theoremstyle{plain}
\newtheorem{theorem}{Theorem}[section]
\newtheorem*{theorem*}{Theorem}
\newtheorem{lemma}[theorem]{Lemma}
\newtheorem*{lemma*}{Lemma}
\newtheorem{cor}[theorem]{Corollary}

\newtheorem{prop}[theorem]{Proposition}
\newtheorem*{prop*}{Proposition}

\theoremstyle{definition}
\newtheorem{define}[theorem]{Definition}

\newtheorem*{example*}{Example}
\newtheorem{claim}[theorem]{Claim}

\theoremstyle{remark}
\newtheorem{remark}{Remark}
\newtheorem*{remark*}{Remark}

\usepackage{todonotes}
\usepackage[numbers]{natbib}
\usepackage{amscd}
\usepackage{tikz}

\usepackage[textwidth=6in,left=1.25in,right=1.25in]{geometry}

\newcommand{\FPP}{first-passage percolation}
\newcommand{\FPT}{first-passage time}

\newcommand{\T}{T} 
\newcommand{\Heff}{H} 
\newcommand{\Weight}{W} 

\def\LHS{left-hand side}

\title{Variational formula for the time-constant of first-passage percolation}
\author{Arjun Krishnan}
\address
{Courant Institute of Mathematical Sciences\newline
\indent New York University\newline
\indent 251 Mercer Street\newline
\indent New York, NY-10012
\indent United States of America}
\curraddr
{\newline
University of Utah\newline
\indent 155 South 1500 East, JWB 209\newline
\indent Salt Lake City, UT 84112
\indent United States of America
}
\date{\today}
\email{arjunkc@gmail.com}
\usepackage{fancyhdr}
\fancypagestyle{plain}{
  \fancyhead{}
  
  \cfoot{}
}
\usepackage[urlcolor=blue,colorlinks=true,linkcolor=blue,citecolor=blue]{hyperref}

\newcommand{\ack}{\section*{Acknowledgements}}

\begin{document}

\begin{abstract}
We consider first-passage percolation with positive, stationary-ergodic weights on the square lattice $\Z^d$. Let $T(x)$ be the first-passage time from the origin to a point $x$ in $\Z^d$. The convergence of the scaled first-passage time $T([nx])/n$ to the time-constant as $n \to \infty$ can be viewed as a problem of homogenization for a discrete Hamilton-Jacobi-Bellman (HJB) equation. We derive an exact variational formula for the time-constant, and construct an explicit iteration that produces a minimizer of the variational formula (under a symmetry assumption). We explicitly identify when the iteration produces correctors.
\end{abstract}

\maketitle

\tableofcontents
\addtocontents{toc}{\protect\setcounter{tocdepth}{1}} 
\thispagestyle{plain}

\section{Introduction}
\label{sec:introduction}
\subsection{Overview}
\label{sec:intro-fpp}
First-passage percolation is a growth model in a random medium
introduced by~\citet{hammersley_first-passage_1965}. We consider the case where the random
medium consists of positive edge weights attached to the edges of the
undirected nearest-neighbor graph on the cubic lattice $\Z^d$.

It's useful to think of \FPP~as an optimal-control problem
(see~\citet{krishnan_variational_2014-1} for more details), and so we'll
define the set of \emph{control directions}
\begin{equation}
  A := \{\pm e_1,\ldots,\pm e_d\},
  \label{eq:control-directions}
\end{equation}
where $e_i$ are the canonical unit basis vectors for the lattice
$\Z^d$.
   
Let $(\W,\mathcal{F},\Prob)$ be a probability space. The weights will
be given by a function $\tau\colon \Z^d \times A \times \W \to \R$,
where $\tau(x,c,\w)$ refers to the weight on the edge from $x$ to
$x+c$. Let the function $\tau(x,c,\w)$ be stationary-ergodic
(see~\Defref{def:stationary-ergodic-process}) under translation by
$\Z^d$. We will drop reference to the event $\w$ when it plays no role
in the arguments.
   
A path connecting $x$ to $y$ is a (possibly infinite) ordered set of
nearest-neighbor vertices:
\begin{equation}
  \gamma_{x,y} = \{x=v_0,\ldots,v_{n-1}=y\}.
  \label{eq:generic-path-gamma-as-set-of-vertices}
\end{equation}
The weight or total time of the path is
\[
  \Weight(\gamma_{x,y}) := \sum_{i=0}^{n-1} \tau(v_i,v_{i+1}-v_i).
\]
The first-passage time from $x$ to $y$ is the infimum of the weight
taken over all paths from $x$ to $y$:
\begin{equation*}
  \T(x,y) := \inf_{\gamma_{x,y}} \Weight(\gamma_{x,y}).
\end{equation*}
Since the medium is translation invariant, we will use $\T(x)$ to mean
$\T(x,0)$ unless otherwise specified. For any $x \in \R^d$, define the
scaled~\FPT
\begin{equation}
  \T_n(x) := \frac{\T([nx])}{n}, 
  \label{eq:n-scaling-first-passage-time}
\end{equation}
where $[nx]$ represents the closest lattice point to $nx$ (with some
fixed way to break ties). The law of large numbers for $\T(x)$
involves the existence of the so-called \emph{time constant} $m(x)$
given by
\begin{equation}
  m(x) := \lim_{n \to \infty} \T_n(x).
  \label{eq:time constant}
\end{equation}
In $d=1$, the limit exists since it's simply the ergodic theorem for a
stationary sequence of random numbers. For $d \geq 1$, Kingman's
classical subadditive ergodic theorem~\citep{kingman_ergodic_1968}
shows the existence of $m(x)$ for all $x \in \R^d$. Although one of
the landmark theorems in the field, it gives little quantitative
information about $m(x)$. Proving something substantial about the
time constant has been an open problem for the last several decades.

We rederive a variational formula for the time constant that was
proved in the author's thesis~\citep{krishnan_variational_2014-1}. The
earlier proof relied on the taking discrete first-passage percolation
into the continuum, and using the seminal homogenization results
of~\citet{lions_homogenization_2005} for Hamilton-Jacobi-Bellmann
(HJB) PDEs. In contrast, the proof in this paper is completely
discrete, and no longer relies on viscosity solution theory or any
other machinery from the continuum. The formula is most conveniently
written in terms of the dual-norm of $m(x)$, which we will call the
effective Hamiltonian  $\Heff(p)$ of~\FPP, because of its connection
with homogenization theory. 

We explore properties of the variational representation by
constructing an explicit algorithm to produce a minimizer of the
formula, thereby computing the effective Hamiltonian and
time constant. The algorithm is proved to converge under a symmetry
assumption. The algorithm was designed to produce correctors, a
special kind of minimizer (see \Secref{def:discrete-corrector}). It
turns out that it always produces minimizers, but sometimes fails to
produce correctors. It fails to do so in a very explicit way, and this
sheds some light on the existence of correctors problem in stochastic
homogenization.

\subsection{First-Passage Percolation as a Homogenization Problem}
\label{sec:fpp-as-homo-problem}
Since the \FPT~$\T(x)$ is an optimal-control problem, it has a dynamic
programming principle (DPP) which says that
\[
  \T(x) = \min_{\alpha \in A} \{ \T(x+\alpha) + \tau(x,\alpha) \}.
\]
We can rewrite the DPP as a difference equation in the so-called
\emph{metric} form of the HJB equation. Assuming $\tau(x,\alpha)$ is
positive, we have
\begin{equation*}
  \sup_{\alpha \in A} \left\{ -\frac{(\T(x+\alpha) - \T(x))}{\tau(x,\alpha)} \right\} = 1. 
\end{equation*}
Let's imagine that we were somehow able to extend $\T(x)$ as a smooth
function on~$\R^d$. Taylor-expand $\T(x)$ at $[nx]$ to get
\begin{equation}
  \sup_{\alpha \in A} \left\{ -\frac{D\T([nx])\cdot \alpha + 1/2 (\alpha \cdot D^2\T(\xi)\alpha) }{\tau([nx],\alpha)} \right\} = 1, 
  \label{eq:discrete-HJB-manipulation-into-homo-problem}
\end{equation}
where $D^2T(\xi)$ is the Hessian at $\xi \in \R^d$. Introducing the scaled \FPT~$\T_n(x)$ into~\Eqref{eq:discrete-HJB-manipulation-into-homo-problem}, we get
\begin{equation}
  \sup_{\alpha \in A} \left\{ -\frac{D\T_n(x)\cdot \alpha}{\tau([nx],\alpha)} \right\} + O(n^{-1})= 1. 
  \label{eq:discrete-stochastic-homo-problem-from-intro}
\end{equation}

Equation~\eqref{eq:discrete-stochastic-homo-problem-from-intro} is the
discrete version of a stochastic homogenization problem for HJB
equations on $\R^d$. In the continuum, one would go on to prove that
$\T_n(x) \to m(x)$, where $m(x)$ satisfies the metric HJB equation
\[
    \Heff(m(x)) = 1 \quad \forall x \in \R^d, \quad m(0) = 0. 
\]
Then, one would obtain a variational representation of the effective
Hamiltonian~$H$. In~\FPP, $H(p)$ is the dual norm of $m(x)$.

There are several such results for stochastic Hamilton-Jacobi
equations on $\R^d$, beginning
with~\citet{souganidis_stochastic_1999},
and~\citet{rezakhanlou_homogenization_2000} for superquadratic
Hamiltonians. The concurrent, seminal
papers~\cite{lions_homogenization_2005,kosygina_stochastic_2006} proved very general
homogenization results for HJB equations with second-order viscous
terms, but with different approaches and assumptions. The more recent
work of~\citet{armstrong_stochastic_2012,armstrong_stochastic_2013} focus specifically on metric
Hamiltonians like the one for \FPP; the former considers the
non-viscous problem and the latter considers the viscous problem.

Armstrong et al.\  \citet{armstrong_error_2012} made the following observation: since
$\T(x,y)$ induces a random metric on the lattice, it's reasonable to
believe that there ought to be some relation to metric HJB
equations. This is exactly what we prove.

\subsection{Background on the Time Constant}
\label{sec:background-on-time constant}
We give a brief overview of results about the time constant in
\FPP. Cox and Durrett \citet{cox_limit_1981} proved a celebrated result about the
relationship between the time constant and the so-called limit shape
of~\FPP. It's a ``uniform in all directions'' version of Kingman's
result. Suppose the edge weights are i.i.d., and let
\begin{equation*}
  \hat{R}(x,t) := \{y \in \R^2 : \T([x],[y]) \leq t \}
\end{equation*}
be the \emph{reachable set} starting from $x$. Under a moment
assumption on the edge weights, they prove that for each $\e > 0$
\[
    \{x \colon m(x) \leq 1 - \e \} \subset t^{-1} \hat{R}(0,t) \subset \{x \colon m(x) \leq 1 + \e\} \textrm{ as } t \to \infty  ~\almostsurely.
\]
Hence, the sublevel sets of the time constant
\[
  B_0 = \{x \colon m(x) \leq 1 \}
\]
can be thought of as the limit shape
of~\FPP. Boivin\,\citet{boivin_first_1990} proved the Cox-Durrett result for
stationary-ergodic edge weights; we rely on his more general result in
this paper and state it precisely
in~\Secref{sec:proofs-related-to-discrete-variational-formula}. Despite
these strong existence results on the time constant and limit shape,
surprisingly little else is known in sufficient
generality~\citep{van_den_berg_inequalities_1993}.

The following is a selection of results when the weights are
i.i.d. It's known that $m(e_1) = 0$ iff $F(0) \geq p_T$, where $p_T$
is the critical probability for bond percolation on
$\Z^d$~\citep{kesten_aspects_1986}.
Durrett and Liggett \citet{durrett_shape_1981} described an interesting class of examples
where $B_0$ has flat spots, and so indeed $B_0$ cannot be the
euclidean ball. Marchand \citet{marchand_strict_2002} and
subsequently Auffinger and Damron  \citet{auffinger_differentiability_2013} have recently
explored some aspects of this class of examples in great detail. Exact
results for the limit shape are only available for ``up-and-right''
directed percolation with special edge weight
distributions~\citep{seppalainen_exact_1998,johansson_shape_2000}. In
fact, Johannson~\citet{johansson_shape_2000} not only obtains the limit shape
for directed percolation, but also (loosely speaking) shows that
\[
  \T(n x) \sim n m(x) + C n^{1/3} \xi,
\]
where $C$ is a constant, and $\xi$ is distributed according to the (GUE) Tracy-Widom
distribution. Because of this and other physical
reasons~\citep{krug_universality_1988}, the fluctuations of \FPP~are
thought to be in the so-called KPZ universality class when $m(x)$ is
strictly convex.

Several theorems can be proved assuming properties of the limit
shape. For example, results about the fluctuations of $\T(x)$ can be
obtained if it's known that the limit shape has a ``curvature'' that's
uniformly
bounded~\citep{auffinger_differentiability_2013,newman_surface_1995}. Chatterjee
and Dey \citet{chatterjee_central_2013} prove Gaussian fluctuations
for \FPP\ in thin cylinders under the hypothesis that the limit shape
is strictly convex in the $e_1$-direction. Therefore, properties like
strict convexity, regularity, or the curvature of the limit shape are
of great interest.

We suggest the lecture notes of~Kesten \citet{kesten_aspects_1986} and the
review papers by~Grimmett and Kesten \citet{grimmett_percolation_2012}
and~Blair-Stahn \citet{blair-stahn_first_2010} for a more exhaustive survey of the
many aspects of first-passage percolation.

\section{Main Results}
\label{sec:main-results}
\subsection{Notation}
$\R^+$ and $\Z^+$ refer to the nonnegative real numbers
and integers respectively.

The expectation of a random variable $X$ will be written as $\E[X]$ or
$\int X \Prob(d\w)$. The expectation of $X$ over a set $B$ will be
written as $\E[X,B]$ or $\int_B X \Prob(d\w)$.

$|\cdot|_p$ is the $l^p$ norm on $\R^d$. $|\cdot|$ without a subscript
will mean the euclidean norm. $x\cdot y$ is the usual dot product on
$\R^d$ between $x$ and $y$. $L^p(\W)$ refers to the space of functions
on $\W$ with the usual $\Norm{\cdot}{p}$ norm.

The Lipschitz norm of a function $f\colon \Z^d \to \R$ is defined as
\begin{equation}
  \Lip(f) := \inf \bigl\{ C : |f(x)-f(y)| \leq C |x-y|_{1} ~\forall\, x,y \in \Z^d \bigr\}.
  \label{eq:discrete-lipschitz-norm}
\end{equation}

When convenient, we will use $a \land b = \min\{a,b\}$.

The initialism DPP stands for dynamic programming principle, and HJB
stands for Hamilton-Jacobi-Bellmann.

Paths are usually called $\gamma_{x,y}$ where $x$ and $y$ are its
starting and ending points. When it is clear from context or
irrelevant to the argument, we
may either omit the endpoint of a path ($\gamma_x$), or both the start- and
endpoints ($\gamma$) from the notation. The $i$\textsuperscript{th} vertex in a
path is called $\gamma_x(i)$. 

$d(\gamma)$ is the $l_1$ length of a path, and $\Weight(\gamma)$ is its weight.
When the path is fixed and understood, $\Weight(\gamma(0),\gamma(k))$ will also
be used to represent the weight of the path from vertex $\gamma(0)$ to vertex $\gamma(k)$.

Constants called $C$ in proofs are mutable; i.e., they may change from
line-to-line.

\subsection{Main Results}
Our main result is a variational formula for the time constant $m(x)$
defined in~\Eqref{eq:time constant}. To state this result, we define
the edge weight process $\tau(x,c,\w)$ and a discrete Hamiltonian
for~\FPP.

Let $\Z^d$ act on $\W$ through a family of invertible
measure-preserving maps
\begin{equation}
  \{V^{x} \colon \Omega \to \Omega\}_{x \in \Z^d}.
  \label{eq:translation-group-of-operators}
\end{equation}
These are called translation operators and satisfy
\[ 
  V^{x+y} = V^x \circ V^y \quad \forall~x,y \in \Z^d.
\]

A function (or process) $f\colon \Z^d \times \Omega \to \R$ is said to
be stationary if it satisfies
\begin{equation}
  f(x + y,\omega) = f(x,V^y\omega) \quad \forall~ x,y \in \Z^d.
  \label{eq:stationary-process-definition}
\end{equation}

We say $B \in \mathcal{F}$ is an invariant set if it satisfies
$V^y B = B$ for any $y \in \Z^d \backslash \{\vec{0}\}$ where
$\vec{0}$ is the origin. The family of maps $\{V^x\}_{x \in \Z^d}$ is
called (totally) ergodic if invariant sets are either null or have
full measure.

\begin{define}[Stationary-Ergodic Process]\label{def:stationary-ergodic-process}
  A function $f(x,\w)$ is called \emph{sta\-tion\-ary ergodic} if it's stationary
  on $\Z^d$, and the family of translation operators
  $\{V^x\}_{x \in \Z^d}$ is (totally) ergodic.
\end{define}

Let $A$ be the set of control directions defined
in~\Eqref{eq:control-directions}. Let the edge weights
$\tau\colon\Z^d \times A \times \W \to \R$ be
\begin{enumerate}
\item (essentially) bounded above and below; i.e.,
  \begin{equation}
    \label{eq:basic-assumption-on-edge-weights}
    \begin{split}
      0 <~& a = \essinf_{x,\alpha,\w} \tau(x,\alpha,\w), \\
      & b = \esssup_{x,\alpha,\w} \tau(x,\alpha,\w) < \infty,
    \end{split}
  \end{equation}
\item stationary-ergodic on $\Z^d$, and
\item on the undirected graph on $\Z^d$,
  \begin{equation}
    \tau(x,\alpha,\w) = \tau(x + \alpha, - \alpha,\w).
    \label{eq:algorithm-undirected-edge-weight-assumption}
  \end{equation}
\end{enumerate}

Under this assumption, the time constant $m(x)$ is a convex,
one-homogeneous function in $x$, that's zero iff $x=0$. Let $\Heff(p)$
be the dual norm of $m(x)$ on $\R^d$, defined as usual by
\begin{equation*}
  \Heff(p) = \sup_{m(x) = 1} p\cdot x.
\end{equation*}
$\Heff$ is called an \emph{effective Hamiltonian} for~\FPP~due to the
homogenization connection.

\begin{define}[Discrete derivative]
  For a function $\phi\colon\Z^d \to \R$, let
  \[
    \mathcal{D}_{\alpha} \phi(x) = \phi(x+\alpha) - \phi(x)
  \]
  be its discrete derivative at $x \in \Z^d$ in the direction
  $\alpha \in A$. $\mathcal{D}\phi$ refers to the vector of discrete
  derivatives $\{\mathcal{D}_{\alpha}\phi\}_{\alpha \in A}$. When
  $\phi \colon \W \to \R$,
  $\mathcal{D}_{\alpha}\phi(\w) = \phi(V^{\alpha}\w) - \phi(\w)$. 
\end{define}

Let the discrete Hamiltonian for \FPP~ be
(see~\Secref{sec:fpp-as-homo-problem} for the motivation)
\begin{equation}
  \mathcal{H}(\phi,p,x,\w) = \sup_{\alpha \in A} \left\{ \frac{-\mathcal{D}_{\alpha}\phi(x,\w) - p\cdot \alpha}{\tau(x,\alpha,\w)} \right\}.
  \label{eq:discrete-hamiltonian-1}
\end{equation}
Define the set of Lipschitz functions with stationary, mean-zero
derivatives:
\begin{equation}
  S:= \left\{ 
    \phi:\Z^d \times \Omega \to \R ~\colon
  \begin{split}  
      & \mathcal{D}\phi(x + z,\w) = \mathcal{D}\phi(x,V^z\w), ~\forall x,z \in \Z^d, \\
      & \Norm{\mathcal{D}_{\alpha}\phi(0)}{\infty} < \infty, \,\E[\mathcal{D}_{\alpha}\phi(0)] = 0 ~\forall~ \alpha \in A .\\
    \end{split}
  \label{eq:discrete-set-S-for-variationalf-formula}
\right\}.
\end{equation}
\begin{theorem}\label{thm:discrete-variational-formula-for-H-bar}
  The effective Hamiltonian $\Heff(p)$ is given by the variational
  formula
  \begin{equation*}
    \Heff(p) = \inf_{\phi \in S\strut} \esssup_{\w \in \W\vphantom{^d}} \sup_{x \in \Z^d} \mathcal{H}(\phi,p,x,\w) . 
  \end{equation*}
\end{theorem}
\begin{remark}
  The supremum over $x$ can be omitted in the formula. This is due to
  translation invariance and ergodicity.
  \label{rem:the-sup-over-x-can-be-dropped}
\end{remark}

Next, we investigate some properties of the minimizers of the formula.
\begin{cor}\label{cor:inf-sup-inequality-for-limiting-Hamiltonian}
  For each
  $\phi \in S$~\Eqref{eq:discrete-set-S-for-variationalf-formula},
  \begin{equation*}
    \inf_{x \in \Z^d} \mathcal{H}(\phi,p,x,\w) \leq \Heff(p) \leq \sup_{x \in \Z^d} \mathcal{H}(\phi,p,x,\w) \quad \almostsurely
  \end{equation*}
\end{cor}

\begin{define}[Discrete corrector]
  If $\phi \in S$ is such that
  \[
    \mathcal{H}(\phi,p,x,\w) = C \quad \almostsurely
  \]
  for some constant $C$, $\phi$ is called a corrector for the
  variational formula.
  \label{def:discrete-corrector}
\end{define}
If $\phi$ is a corrector,
then~\Corref{cor:inf-sup-inequality-for-limiting-Hamiltonian} tells us
that it's a minimizer of the variational formula. This definition is
consistent with the definition of corrector in continuum
homogenization
theory~\citep{lions_homogenization_1987,lions_correctors_2003}; i.e.,
it's a function that solves the discrete cell problem
(see~\Secref{sec:outline-variational-formula}).

Correctors are useful minimizers to have since they appear as
first-order corrections in the multiscale expansion in
homogenization. In~\FPP, their properties are connected to the
behavior of infinite geodesics and Busemann functions
(see~\Secref{sec:comments-busemann-functions}). Although correctors
always exist when the medium is periodic, it's known that they don't always
exist in general stationary ergodic media~\citep{lions_correctors_2003}. The
compactness argument in~\Lemref{lem:compactness-argument-for-lower-bound} tells
us that minimizers of the formula exist, but sheds no light on whether or not
they're correctors. This inspired us to construct an explicit algorithm to
produce minimizers of the formula, and see if it would produce correctors.
However, we were only able to prove our theorem in a simplified, but reasonably
nontrivial scenario: we assume that the generating translation operators are
the same in all the directions; i.e.,
\begin{equation}
  V^{e_1} = \cdots = V^{e_d} = V .
  \label{eq:symmetry-assumption-on-medium}
\end{equation}
This means that for each $\w$ the function $\tau(\,\cdot\,,\cdot\,,\w)$ is
constant along the hyperplanes
$\{ x \in \Z^d \colon \sum_{i=1}^d x_i = z \}$ for each $z \in
\Z$.
The set $S$ in~\Eqref{eq:discrete-set-S-for-variationalf-formula} is
tremendously simplified by this assumption.

Redefine the discrete Hamiltonian for $t \in \R$, $p \in \R^d$ to be
\begin{equation*}
  \mathcal{H}_{\rm sym}(t,p,\w) := \sup_{\alpha \in A^+} \frac{|t + p\cdot\alpha|}{\tau(0,\alpha,\w)} . 
\end{equation*}
\begin{prop}\label{prop:variational-formula-in-symmetric-situation}
  Assuming~\Eqref{eq:symmetry-assumption-on-medium}, the variational
  formula becomes
  \begin{equation}
    \Heff(p) = \inf_{f \in F} \esssup_w \mathcal{H}_{\rm sym}(f(\w),p,\w),
  \end{equation}
  where
  \begin{equation}\label{eq:set-S-under-symmetry-assumption}
    F := \left\{ f:\W \to \R ,~ \E[f] = 0,\, \Norm{f}{\infty} < \infty \right\}.
  \end{equation}
\end{prop}

The algorithm is a map $\mathcal{I}:F \to F$ defined in~\Secref{sec:explicit-algorithm-to-produce-a-minimizer}.
\begin{theorem}  \label{thm:algorithm-convergence}
  Let $\{f_n\}_{n=0}^{\infty}$ be the sequence obtained by iterating
  the algorithm on an initial point $f_0 \in F$. Let $d_n =
  \esssup_{\W} \mathcal{H}_{\rm sym}(f_n(\w),p,\w) - \E[ \mathcal{H}_{\rm sym}]$. There are three
  possibilities for the algorithm:
\begin{enumerate}
  \item If it terminates in a finite number of steps with $d = 0$, we have a minimizer that's a corrector.
  \item If it terminates in a finite number of steps with $d > 0$, we
    have a minimizer that's not a corrector.
  \item If it does not terminate, we produce a corrector in the
    limit. That is, $f_n \to f_{\infty}$ in measure and
    $\mathcal{H}_{\rm sym}(f_{\infty},p,\w) = \Heff(p)~\almostsurely$.
  \end{enumerate}
\end{theorem}

We give examples where the algorithm produces minimizers that are
correctors or noncorrectors, depending on the initial
point $f_0$.  See~\Secref{sec:explicit-algorithm-to-produce-a-minimizer} for
more discussion and proofs.

\section{Discussion}\label{sec:some-comments-on-the-setup}
\subsection{Directed~First-Passage Percolation and Other Problems}
We've formulated the problem so that it applies to
\begin{itemize}
\item first-passage percolation on the undirected nearest-neighbor
  graph of $\Z^d$, and
\item site first-passage percolation (weights are on the vertices of
  $\Z^d$) if
  \[
    \tau(x, \alpha) = \tau(x) \quad \forall x \in \Z^d,\ \alpha \in A.
  \]
\end{itemize}
These are by no means the most general problems that come under this
framework. Specializing to nearest-neighbor first-passage percolation
has mostly been a matter of convenience.

If $A = \{e_1,\ldots,e_d\}$ and we consider $\T(0,x)$, we get directed
first-passage percolation. Versions of the theorems
in~\Secref{sec:main-results} do indeed hold for such $A$. If $A$ is
enlarged to allow for long-range jumps---and very large jumps are
appropriately penalized---we obtain long-range percolation. We avoid
handling the subtleties of these examples in this paper.

The $(d+1)$-dimensional directed random polymer assigns a random cost to
randomly chosen paths in $\Z^d$. At zero temperature, this corresponds
to last-passage percolation. Variational formulas for polymer models
at zero and finite-temperature with a large class of edge weights and
control directions $A$ have been proved
by~\citet{georgiou_variational_2013}.

\subsection{More General Edge Weights}
Our homogenization theorem
(\Thmref{thm:homogenization-theorem-for-point-to-line-passage-time})
depends on Boivin's theorem
(\Thmref{thm:boivins-stationary-cox-durrett}) on the limit shape, and
the ``uniform'' ergodic theorem (\Thmref{thm:uniform-ergodic-theorem}),
which only require moment assumptions on $\tau(0,\alpha,\w)$. Hence
it is of some interest to remove the boundedness assumption on the
edge weights in~\Eqref{eq:basic-assumption-on-edge-weights}. The key
step that one ought to generalize
is~\Propref{prop:lipschitz-estimate-for-stationary-problem}; we explicitly
use the bounds on the edge weights there. 

The discrete Hamiltonian we considered
in~\Eqref{eq:discrete-hamiltonian-1} has some drawbacks: It would take
some effort  to make sense of it when $\tau(x,c,\w) = 0$. It is of
some interest to include this case, since in i.i.d.~\FPP, the limit shape
and time constant are nontrivial as long as
$\Prob(\tau(x,c,\w) = 0) < p_T$, where $p_T$ is the critical
probability for percolation (a result of Kesten mentioned in the introduction). In
contrast, the Legendre-type convex duality and the Hamiltonian
in~\citet{georgiou_variational_2013} does not suffer from this issue.

\subsection{Other Variational Formulas}
Once we posted our preprint on the
arXiv---it now appears in~\citep{krishnan_variational_2014-1}--- the concurrent
but independent work of Georgiou, Rassoul-Agha and Sepp\"al\"ainen~\citet{georgiou_variational_2013} appeared. Their ideas
originate in the continuum homogenization result
of~\citet{kosygina_stochastic_2006}, and the works
of~\citet{rosenbluth_quenched_2008,rassoul-agha_quenched_2014,rassoul-agha_quenched_2013} on large-deviation principles
for the random walk in a random environment.

It is interesting to note that quite coincidentally, our results and
those of~Georgiou et al.\ \citet{georgiou_variational_2013} almost exactly parallel the
development of stochastic homogenization results in the
continuum. Lions and Souganidis \citet{lions_homogenization_2005} published their
homogenization results in 2005, using the classical cell problem idea
and the viscosity solution framework. Concurrently and independently
in 2006, Kosygina et al.\ \citet{kosygina_stochastic_2006} published their
homogenization result. In contrast
to~\citet{lions_homogenization_2005}, their proof technique has the
flavor of a duality principle and has a minimax theorem at its
core. Both our results and those of~\citet{georgiou_variational_2013}
are discrete versions of~\citet{lions_homogenization_2005}
and~\citet{kosygina_stochastic_2006}, respectively.


\subsection{Busemann Functions and Geodesics}
\label{sec:comments-busemann-functions}
The minimizers of the formula
in~\Thmref{thm:discrete-variational-formula-for-H-bar} are very
closely related to Busemann functions, which were originally
introduced into \FPP\  in~Newman \citet{newman_surface_1995} (although he
didn't call them that). Busemann functions encode information about
infinite geodesics; the asymptotic slope of infinite geodesics and
their coalescence properties are closely related to regularity
properties of the time constant. If correctors exist, they can be used
to construct generalized versions of Busemann
functions. See~\cite{MR2114988,MR2462555,MR3152744,licea_superdiffusivity_1996,MR2115045} for more on geodesics
in~\FPP. See~\cite{georgiou_geodesics_2015,georgiou_stationary_2015} for more on the connection
between the minimizers of similar variational formulas, Busemann
functions, and geodesics.

\section{Outline of Proof}
\label{sec:outline-of-paper}
\subsection{Variational Formula}
\label{sec:outline-variational-formula}
\mbox\newline
\noindent \textbf{Step 1. Homogenization theorem and dual problem.}
Suppose that we have \emph{running costs}
$\lambda:\Z^d \times A \to \R$. Summing $\lambda$ along a path assigns
an action or cost to it.

Given a function $\phi\colon\Z^d \to \R$ (usually called the
\emph{terminal cost}), consider the variational problem
\begin{equation}\begin{split}
  \mu(x,t) = \inf_{\gamma_{x,y}} \Biggl\{&
  \sum_{i=0}^{d(\gamma_{x,y})-1}
  \lambda\bigl(\gamma_{x,y}(i),\gamma_{x,y}(i+1) -
  \gamma_{x,y}(i)\bigr) + \phi(y) : \\[-\jot]
&\Weight(\gamma_{x,y}) \leq t \Biggr\}.
  \label{eq:discrete-finite-time-horizon-problem}
\end{split}
\end{equation}
This is usually called a finite time-horizon problem in control
theory. Although one can prove general homogenization theorems for
general random running costs $\lambda(x,\alpha)$, we will restrict our
attention to the case when $\lambda(x,\alpha) = p\cdot\alpha$.

\begin{define}[Dual problem]\label{def:discrete-cell-problem}
  Given a terminal cost function $\phi \colon \Z^d \to \R$, the \emph{dual
  problem} to~\FPP~is
  \[
    \mu(x,t) = \inf_{\gamma_{x,y}} \left\{ p\cdot (y-x) + \phi(y)
      : \Weight(\gamma_{x,y}) \leq t \right\}.
  \]
\end{define}
In the context of first- and last-passage percolation, $\mu(x,t)$ is
sometimes called the point-to-line passage
time~\citep{georgiou_variational_2013}. But it's not \emph{quite} the
passage time to any line. So we'll prefer to simply call this the dual
problem. There is a convex duality relationship between the $\mu(x,t)$
and the first-passage time $T(x)$ in the limit $t \to \infty$.

\begin{theorem}[Homogenization theorem for the dual problem]\label{thm:homogenization-theorem-for-point-to-line-passage-time}
  Consider the dual problem $\mu(x,t)$ defined
  in~\Defref{def:discrete-cell-problem}. If $\phi \in S$ (see~\Eqref{eq:discrete-set-S-for-variationalf-formula}, then for all $x \in \Z^d$,
  \[
    \lim_{t \to \infty} \frac{\mu(x,t)}{t} = H(p) \quad \almostsurely,
  \]
  where $H(p)$ is the effective Hamiltonian $($the dual norm of the
  time constant $m(x))$.
\end{theorem}

\par\noindent \textbf{Step 2. Homogenization theorem and dual problem.}
A major tool in the theory of HJB equations is the comparison
principle for sub- and super-solutions. It is a consequence of the
following DPP for the dual problem.

\begin{prop}\label{prop:DPP-for-discrete-variational-problem-for-comparison-principle}
  With the convention that $\mu(x,t) = +\infty$ when $t < 0$, the DPP
  for the finite time-horizon problem $\mu(x,t)$ with terminal cost
  $\phi(x)$ and running costs $\lambda(x,\alpha)$ takes the form
  \[ 
  \mu(x,t) =
    \inf_{\alpha \in A} \{ \mu(x+c,t-\tau(x,c)) + \lambda(x,\alpha) \} \land \phi(x).\\
  \]
\end{prop}

The comparison principle that we prove is a simplified version of the
much more general idea for sub- and super-solutions of HJB equations
(see~\citet{bardi_optimal_1997}, for example) that suffices for our
purposes. Propositions~\ref{prop:comparison-principle}
and~\ref{prop:other-inequality-comparison-principle} are the sub- and
supersolution versions respectively. Propositions~\ref{prop:DPP-for-discrete-variational-problem-for-comparison-principle},~\ref{prop:comparison-principle},~\ref{prop:discrete-dpp-stationary-problem} and~\ref{prop:other-inequality-comparison-principle} all hold for general costs
$\lambda(x,\alpha)$ and not just for the special case of
$\lambda(x,\alpha) = p \cdot \alpha$; hence we'll write the
Hamiltonian as $\mathcal{H}(\phi,x)$.

\begin{prop}\label{prop:comparison-principle}
  Consider the finite time-horizon problem
  in~\Eqref{eq:discrete-finite-time-horizon-problem} with terminal
  cost $\phi$. Then,
  \[ 
    \mu(x,t) \geq \phi(x) - t \sup_{x \in \Z^d} \mathcal{H}(\phi,x) \quad \forall x \in \Z^d,~ t \in \R^+.
  \]
\end{prop}
The following upper bound for the effective Hamiltonian follows from a
simple argument using~\Propref{prop:comparison-principle} and~\Thmref{thm:homogenization-theorem-for-point-to-line-passage-time}.
\begin{lemma}\label{lem:upper-bound-for-variational-formula-using-comparison-principle}
  Let $\phi \in S$, where $S$ is defined in~\Eqref{eq:discrete-set-S-for-variationalf-formula}. Then the effective Hamiltonian satisfies
  \[
    H(p) \leq \sup_x \mathcal{H}(\phi,p,x,\w)  \quad\almostsurely.
  \]
\end{lemma}

\par\noindent\textbf{Step 3. Lower bound by constructing a minimizer}
In periodic homogenization, a properly rescaled version of $\mu(x,t)$
converges to a function $u(x)$ (in an appropriate sense). This
function $u(x)$ is called a corrector since it solves the so-called
cell problem, $\mathcal{H}(u,p,x) = H(p)$. From the upper bound
in~\Lemref{lem:upper-bound-for-variational-formula-using-comparison-principle},
it follows that such a function $u(x)$ is a minimizer of the
variational formula. 

The cell problem was introduced systematically into periodic
homogenization by~\cite{MR503330}, and it has since become a mainstay of homogenization theory. It was first used to analyze stochastic HJB
equations by~\cite{lions_homogenization_1987}.

We cannot produce a corrector for the variational formula using
$\mu(x,t)$ in our general stationary-ergodic setting. Nevertheless, a
version of the argument still produces a minimizer.

\begin{lemma}\label{lem:compactness-argument-for-lower-bound} 
    For each $p > 0$, there exists a function $u(x,\w) \in
    S$ such that 
    \[ 
        \sup_x \mathcal{H}(u,p,x,\w) \leq H(p) \quad\almostsurely.  
    \]
    The function $u$ is a minimizer of the variational formula.
\end{lemma}

We cannot work with $\mu(x,t)$ directly since it does not have good
stationarity properties due to the discrete nature of our problem. So
we introduce $\nu_{\e}(x)$, the so-called stationary (or discounted)
version of the dual-problem. By proving (one-half of) an
Abelian-Tauberian theorem, we will relate the limits of $\e
\nu_{\e}(x)$ and $t^{-1} \mu(x,t)$ as $\e \to 0$ and $t \to \infty$.
This will help produce the minimizer of the variational formula. 

\begin{define}[Stationary dual problem]
  For any $\e > 0$, the stationary dual problem is defined as 
  \[
      \nu_{\e}(x) = \inf_{\gamma_x} \sum_{i = 0}^{\infty} e^{-\e W(\gamma_x(0),\gamma_{x}(i))} p \cdot (\gamma_x(i+1) - \gamma_x(i)),
  \]
  \label{def:stationary-dual-problem}
\end{define}

The stationary dual problem has the following discrete DPP.
\begin{prop}
    \begin{equation}
        \nu_{\e}(x) = \inf_{\alpha \in A} \left( \alpha \cdot p + e^{-\e \tau(x,\alpha)} \nu_{\e}(x + \alpha) \right).
        \label{eq:discrete-dpp-stationary-problem}
    \end{equation}
    \label{prop:discrete-dpp-stationary-problem}
\end{prop}

With a little manipulation of the DPP, we may view $\nu_{\e}(x)$ as satisfying a discrete HJB equation in stationary form. 
\begin{prop}\label{prop:approximate-discrete-HJB-for-stationary-problem}
  For all $x \in \Z^d$, there is a constant $C > 0$ (independent of $\e$ and $\w$) such that
  \begin{equation}
      -C \e \leq \e \nu_{\e}(x) + \mathcal{H}(\nu_{\e},p,x) \leq C \e,
      \label{eq:approximate-discrete-HJB-for-stationary-problem}
  \end{equation}
  where $\mathcal{H}$ is the discrete Hamiltonian~\Eqref{eq:discrete-hamiltonian-1}.
\end{prop}

We will take $\e \to 0$
in~\Eqref{eq:approximate-discrete-HJB-for-stationary-problem}. To do
this, we need the following Lipschitz estimate on $\nu_{\e}$.
\begin{prop}[Derivative bound for stationary problem]\label{prop:lipschitz-estimate-for-stationary-problem}
  The stationary problem $\nu_{\e}(x,t)$ has a uniform in $\e$
  Lipschitz bound; i.e., $\exists\, C > 0$ s.t. for all $\e \leq 1$,
 \[
     |\nu_{\e}(x) - \nu_{\e}(y)| \leq  C |p|_{\infty} |x-y|_1 \quad \forall~ x,y \in \Z^d
 \]
\end{prop}
\Propref{prop:lipschitz-estimate-for-stationary-problem} says that
$\mathcal{D}\nu_{\e}(x,\w)$ is uniformly bounded in $\e$, and hence
for any fixed $x$, it has a weak limit point $\mathcal{D}u(x,\w)$ in
$L^2(\W)$. This allows us to take a subsequential limit as $\e \to 0$
in~\eqref{eq:approximate-discrete-HJB-for-stationary-problem}.
However, we have yet to identify the limit of $\e \nu_{\e}$
in~\Eqref{eq:approximate-discrete-HJB-for-stationary-problem}. For our
purposes, it's enough to prove one half of an Abelian-Tauberian
theorem for $\e \nu_{\e}$ and $t^{-1} \mu(x,t)$. 
\begin{prop}\label{prop:one-half-of-abelian-theorem}
\begin{equation}
    \liminf_{n \to \infty} \frac{\mu(x,n)}{n} \leq \liminf_{\e \to 0} \e \nu_{\e}(x)
    \label{eq:one-inequality-in-abelian-theorem}
\end{equation}
\end{prop}
Of course, in our situation, it's clear that we ought to have 
\[
  \lim_{\e \to 0} \e \nu_{\e} = \lim_{t \to \infty} t^{-1} \mu(x,t) = -H(p).
\]
But to complete this proof requires a little bit of extra work that's
not necessary for proving the variational formula; hence we will
content ourselves with citing ~\citet{MR1161156},
and~\citet{MR1614615} who do prove that $\e \nu_{\e}$ and $t^{-1}
\mu(x,t)$ have the same limit in slightly different contexts. 

Proposition~\ref{prop:one-half-of-abelian-theorem} allows us to take a
limit $\e \to 0$
in~\Propref{eq:approximate-discrete-HJB-for-stationary-problem} and
prove~\Lemref{lem:compactness-argument-for-lower-bound}.


subsection{Minimizers of the Formula}
Recall the comparison principle
for $\mu(x,t)$ in~\Propref{prop:comparison-principle}.
An almost identical proof, but with a bunch of inequalities reversed,
gives the following proposition.
\begin{prop}\label{prop:other-inequality-comparison-principle}
  Let $\mu(x,t)$ be the finite time-horizon problem with terminal cost
  $\phi$ and running costs $\lambda$. Then,
  \[
    \mu(x,t) - \max(b \inf_x \mathcal{H}(\phi,x),0) \leq \phi(x) - t
    \inf_{x \in \Z^d} \mathcal{H}(\phi,x) \quad \forall \, x,t,
  \]
  where $b$ is the upper bound for the edge weights
  in~\Eqref{eq:basic-assumption-on-edge-weights}.
\end{prop}
Following the same argument as in the proof of the upper bound
in~\Lemref{lem:upper-bound-for-variational-formula-using-comparison-principle} proves
\Corref{cor:inf-sup-inequality-for-limiting-Hamiltonian}.

Corollary~\ref{cor:inf-sup-inequality-for-limiting-Hamiltonian} tells
us that if a function $u \in S$ satisfies $\mathcal{H}(u,p,x\w) = C$,
it is a minimizer of the variational formula. We wanted to explicitly
produce such a special minimizer using an algorithm; however, we were
unable to construct an algorithm that works in full generality. The 
key simplification for the algorithm is
in~\Propref{prop:simplifying-the-set-of-functions-S}, where the set of
functions $S$ in the variational
formula~\Eqref{eq:discrete-set-S-for-variationalf-formula} is
simplified to $F$. It says that under the symmetry assumption, all
functions in $S$ have derivatives that point \emph{only} in the
$\sum_i e_i$-direction. This one-dimensionalizes the problem, in a
sense.

The algorithm simply does the following: Given any function in
$f_0 \in F$, it tries to produce a function $f_1$ such that the
$\esssup$ in the variational formula is decreased. If the algorithm
fails to reduce the $\esssup$, we prove that we must be at a minimizer
of the formula. The strategy of the algorithm is quite general and
ought to be generalizable when the symmetry assumption is
removed.\footnote{To be taken with a pinch of salt---we tried and ran
  into difficulties.} The algorithm and its proof are
in~\Secref{sec:explicit-algorithm-to-produce-a-minimizer}.

\section{Proof of the Variational Formula}
\label{sec:proofs-related-to-discrete-variational-formula}

\subsection{Step 1: \em Homogenization of the dual problem}
We first
prove~\Thmref{thm:homogenization-theorem-for-point-to-line-passage-time},
which says that the dual problem $\mu(x,t)$ homogenizes to $H(p)$. We need the stationary-ergodic version of the classical Cox-Durrett
theorem~\citep{cox_limit_1981} due to
\cite{boivin_first_1990}. For this, we need the following
definitions.

\begin{define}
  The \emph{Lorenz norm} of a function $f$ is defined as
  \[
    \Norm{f}{d,1} := \int_0^1 f^*(s) s^{(1/d)-1}\, ds
  \]
  where $f^* \colon [0,1] \to \R^+$ is the nonincreasing
  right-continuous function that has the same distribution as $|f|$.
  \label{def:lorentz-norm}
\end{define}

\begin{define}[{Reachable set}]
  For $x \in \Z^d$ and $t \in \R^+$, let
  \begin{equation*}
    R(x,t) := \{ y \in \Z^d : \T(x,y) \leq t \} 
  \end{equation*}
  be the set of sites that can be reached from~$x$ within time~$t$. $\hat{R}(x,t)$ is the fattened up version of $R(x,t)$ defined
  in~\Secref{sec:background-on-time constant}.
  \label{def:reachable-set-definition}
\end{define}

\begin{theorem}[\hspace{-.00001pt}{\cite{boivin_first_1990}}]\label{thm:boivins-stationary-cox-durrett}
  Let $\tau(x, \cdot\,, \w)$ be (totally) stationary-ergodic, and let
  $\tau(0,e_i,\linebreak[1]\w)$ have finite Lorentz norm for each
  $i=1,\ldots,d$. Then for each $\e > 0$, almost surely, there exists
  a $t_0(\w)$ large enough such that
  \[
    \{ x \colon \mu(x) \leq 1 - \e \} \subset t^{-1} \hat{R}(0,t)
    \subset \{ x \colon \mu(x) \leq 1 + \e \} \quad \forall t \geq
    t_0(\w).
  \]
\end{theorem}

We also need an ergodic theorem that is ``uniform in all
directions.'' A version of the theorem is neatly proved
in~\cite{rosenbluth_quenched_2008}.

\begin{theorem}[\hspace{-.00001mm}\cite{rosenbluth_quenched_2008}]\label{thm:uniform-ergodic-theorem}
  Consider a function $\phi \colon \Z^d \times \W$. Suppose
  \begin{itemize}
  \item its derivative $\mathcal{D}\phi$ is stationary-ergodic, and
  \item for each $\alpha \in A$,
    $\mathcal{D}_{\alpha}\phi(\w) \in L^{d+\e}(\W)$ for some $\e > 0$.
  \end{itemize}

  Then
  \[
    \lim_{n \to \infty} \sup_{\substack{|z|_1 \leq n,\\ z \in \Z^d}}
    \frac{\phi(z,\w)}{n} = 0 \quad\almostsurely
  \]  
\end{theorem}
\begin{proof}[Proof
  of~\Thmref{thm:discrete-variational-formula-for-H-bar}]
  Let the terminal cost function $\phi \in S$. Although
  $\mu(x,t,\w)$ is not stationary when $\phi \not\equiv 0$, we might
  as well consider $t^{-1} \mu(0,t,\w)$ since
  \begin{align*}
    \limsup_{t \to \infty} \frac{\mu(x,t,\w)}{t} 
    & = \limsup_{t \to \infty} \inf_{y \in \Z^d} \left\{ p \cdot \frac{y-x}{t} + \frac{\phi(y,\w)}{t},\, T(x,y,\w) \leq t \right\}\\
    & = \limsup_{t \to \infty} \inf_{z \in \Z^d} \biggl\{ p \cdot \frac{z}{t} + \frac{\phi(z,V^x \w)}{t} + \frac{\phi(x,\w) - \phi(0,V^x\w)}{t},\\*
    &= \limsup_{t \to \infty} \inf_{z \in \Z^d} \biggl\{\  T(0,z,V^x\w) \leq t \biggr\}\\*
    & = \limsup_{t \to \infty} \frac{\mu(0,t,V^x\w)}{t}
  \end{align*}
  where we've made the change of variable $y-x \to z$ and used the
  stationarity of $\mathcal{D}\phi$ to write
  $\phi(x+z,\w) = \phi(x,\w) - \phi(0,V^x\w) + \phi(z,V^x\w)$. 
    
  Then,
  \begin{align}
    \limsup_{t \to \infty} \frac{\mu(0,t,\w)}{t}
      & = \limsup_{t \to \infty} \inf_{m \in t^{-1} \Z^d } \left\{ p \cdot m + \frac{\phi(tm)}{t} ,\, T(t m) \leq t \right\} \nonumber, \\
      & = \limsup_{t \to \infty} \inf_{s \in \R^d } \biggl\{ p \cdot s + \frac{\phi([ts])}{t} ,\, T([t s]) \leq t \biggr\} \label{eq:homo-thm-change-of-variable} ,\\
      & = \limsup_{t \to \infty\vphantom{R^d}} \inf_{s \in \R^d } \biggl\{ p \cdot s + \frac{\phi([ts])}{t} ,\,  s \in t^{-1} \hat{R}(0,t), \, |s|_1 \leq a^{-1} + 1 \biggr\} \label{eq:homo-thm-limit-uniform-ergodic-thm} , \\
      & = \inf_{s \in \R^d } \{ p \cdot s ,\,  s \in \{m(s) \leq 1\} \} \label{eq:homo-thm-limit-calc-boivins-thm},\\
      & = -H(p), \nonumber
        \label{eq:limit-calculation-for-cell-problem}
  \end{align}
  where we've
  \begin{itemize}
  \item made a change of variable $m \mapsto s$ and expanded the
    domain to $\R^d$ in~\Eqref{eq:homo-thm-change-of-variable},
  \item restricted the infimum
    in~\Eqref{eq:homo-thm-limit-uniform-ergodic-thm} to the set
    $|s|_1 \leq a^{-1} + 1$ since $T([ts]) \geq a |ts|_1 - a$,
  \item bounded $\phi([ts])/t$
    using~\Thmref{thm:uniform-ergodic-theorem} in
    step~\Eqref{eq:homo-thm-limit-uniform-ergodic-thm}, and
  \item used the limit shape theorem
    (\Thmref{thm:boivins-stationary-cox-durrett})
    in~\Eqref{eq:homo-thm-limit-calc-boivins-thm}.
  \end{itemize}
  A similar calculation for the liminf completes the proof. 

\end{proof}

\subsection{Step 2. Comparison principle and upper bound}

First we prove the dynamic programming principle
in~\Propref{prop:DPP-for-discrete-variational-problem-for-comparison-principle}.
\begin{proof}[Proof
    of~\Propref{prop:DPP-for-discrete-variational-problem-for-comparison-principle}]
    If $t < \min_{\alpha \in A} \tau(x,\alpha)$, then no neighbor of
    $x$ can be reached, and $\mu(x,t) = \phi(x)$.  So, assume that at least one neighbor $x + \alpha$ can be reached. Hence,
  \[
    \mu(x,t) \leq \{ \mu(x + \alpha, t - \tau(x,\alpha)) + \lambda(x,\alpha) \} \land \phi(x).
  \]
  Since $\mu(x,t) = + \infty$ for $t < 0$, we may write
  \[
    \mu(x,t) \leq \inf_{\alpha \in A} \{ \mu(x +
    \alpha,t-\tau(x,\alpha)) + \lambda(x,\alpha) \} \land \phi(x).
  \]
  For the opposite inequality, for any $\e > 0$, there is a path
  $\gamma$ such that
  \begin{align*}
      \mu(x,t) & \geq \sum_{i=0}^{d(\gamma_{x,y})-1} \lambda(\gamma_{x,y}(i),\gamma_{x,y}(i+1)-\gamma_{x,y}(i)) + \phi(y) - \e, \\*[-3\jot]
      & \geq \bigl\{\lambda(x,\gamma_{x,y}(2)-\gamma_{x,y}(1)) +
      \mu(\gamma_{x,y}(2),\vphantom{\sum^{d(\gamma_{x,y})-1}}\\[-3\jot]
&\geq \bigl\{\ t-\tau(x,\gamma_{x,y}(2)-\gamma_{x,y}(1))) \bigr\} \land \phi(x) - \e, \vphantom{\sum^{d(\gamma_{x,y})-1}}\\*[-3\jot]
      & \geq \inf_{\alpha \in A} \{ \mu(x,t-\tau(x,\alpha)) +
      \lambda(x,\alpha) \} \land \phi(x) - \e.\vphantom{\sum^{d(\gamma_{x,y})-1}}
  \qedhere\end{align*}
\end{proof}

\begin{proof}[Proof of~\Propref{prop:comparison-principle}]
  We will drop the reference to $p$ in $\mathcal{H}(\phi,p,x)$ and not
  specialize to $\lambda(x,\alpha) = p \cdot \alpha$ in the following. Let 
  \begin{equation*}
    \zeta(x,t) = \phi(x) - t \sup_x \mathcal{H}(\phi,x).
  \end{equation*}
  From the form of $\mathcal{H}$
  (see~\Eqref{eq:algorithm-hsym-form-when-tau-undirected}) it follows
  that $\sup_x \mathcal{H}(\phi,x) \geq 0$ and hence for all
  $ t > 0,~\zeta(x,t) \leq \phi(x)$.

  The proof proceeds by induction on $n$ for $t \in [(n-1)a,na)$. For
  any $n \geq 1$, assume
  \[
    \zeta(x,t) \leq \mu(x,t) \forall x\in\Z^d ,~t < (n-1)a.
  \]
  Let $C = \sup_{x \in \Z^d} \mathcal{H}(\phi,x)$. Then,
  \begin{equation}
    \begin{split}
     & \sup_{\alpha \in A} \left\{ \frac{\zeta(x,t) - \zeta(x+\alpha,t-\tau(x,\alpha)) - \lambda(x,\alpha)}{\tau(x,\alpha)} \right\} \\
     & = \sup_{\alpha \in A} \left\{ \frac{-(\phi(x+\alpha)-\phi(x)) - \lambda(x,\alpha)}{\tau(x,\alpha)} \right\} - C \\
     & = \mathcal{H}(\phi,x) - C \leq 0.
    \end{split}
    \label{eq:comparison-principle-sub-super-solution-display}
  \end{equation}
  Since $\tau(x,\alpha)$ is positive, this implies
  \[
    \zeta(x,t) \leq \inf_{\alpha \in A} \left\{ \zeta(x+\alpha,t-\tau(x,\alpha)) +
      \lambda(x,\alpha) \right\} \land \phi(x).
 \]
Let $t \in [(n-1)a,na)$. There are two cases: suppose first that 
  $t \geq \min_\alpha \tau(x,\alpha)$. Since $\tau(x,\alpha) \geq a$, 
  $t - \tau(x,\alpha) <  (n-1)a$, we may use the induction hypothesis
  to get for all $x \in \Z^d$ 
  \[
    \zeta(x,t) \leq \inf_{\alpha \in A} \left\{ \mu(x+\alpha,t-\tau(x,\alpha)) +
     \lambda(x,\alpha) \right\} \land \phi(x) . 
  \]
  When $t < \min_\alpha \tau(x,\alpha)$, we have 
  $\zeta(x,t) \leq \phi(x) = \mu(x,t)$. Combining the two cases and
  using the DPP
  in~\Propref{prop:DPP-for-discrete-variational-problem-for-comparison-principle}
  completes the induction step. For $n=1$, we have
  $t < a \leq \min_\alpha \tau(x,\alpha)$, and this falls into the previously
  considered case.
\end{proof}

\begin{proof}[Proof of~\Lemref{lem:upper-bound-for-variational-formula-using-comparison-principle}]
  Let $\phi \in S$, where $S$ is defined
  in~\Eqref{eq:discrete-set-S-for-variationalf-formula}. Then
\[
 \sup_x \mathcal{H}(\phi,p,x,\w) < \infty \almostsurely,
\] and the
  comparison principle in~\Propref{prop:comparison-principle} gives
  \[
    \phi(x) - t \sup_x \mathcal{H}(\phi,p,x,\w) \leq \mu(x,t,\w)
    \quad\forall x \in \Z^d \ \almostsurely
  \]
  Divide the inequality by $t$, take a limit as $t\to\infty$,
  use~\Thmref{thm:homogenization-theorem-for-point-to-line-passage-time},
  and rearrange to get
  \[
    \Heff(p) \leq \sup_{x \in \Z^d} \mathcal{H}(\phi,p,x,\w) \quad
    \almostsurely
  \qedhere\]
\end{proof}

\subsection{Lower Bound}
The main idea is to relate the limits of the dual problem $t^{-1} \mu(x,t)$ and its stationary version $\e \nu_{\e}$ as $t \to \infty$ and $\e \to 0$ respectively. It's easier to do so if we modify the definition of $\nu_{\e}$.

\begin{define}[Time parametrization of a path]
    Let $\gamma = (v_i)_{i=1}^{\infty}$ be a path and let $i_k = \lceil W(v_0,v_k)/a \rceil$ for $k \in \Z^+$. Let $y_{\gamma} \colon \Z^+ \to \Z^d$ be the time-parametrization of the path defined by:
    \[
        y_{\gamma}(i) = v_{k}    \quad i_{k} \leq i < i_{k+1}, \quad k \in \Z^+.
    \]
\end{define} 
The time-parametrized path satisfies $W(y_{\gamma}(0),y_{\gamma}(i)) \leq ia$. It jumps from $y_{\gamma}(i)=v_{k}$ to its neighbor $v_{k+1}$ at the first $i$ when $v_{k+1}$ can be reached; that is, 
\[
    y_{\gamma}(i_{k+1}) - y_{\gamma}(i_{k}) = v_{k+1} - v_k \quad k \in \Z^+.
\]
Then, the dual problem maybe defined for all $n \in \Z^+$ as
\[
    \mu(x,n a) = \inf_{\gamma_x} \{ p \cdot y_{\gamma_x}(n) \}.
\]

\begin{define}[Time-parametrized stationary problem]
    For any $\e > 0$, let
    \begin{align*}
        \overline{\nu}_{\e}(x) 
        & = \inf_{\gamma_x} \sum_{i = 0}^{\infty} e^{-\e a i}p \cdot (y_{\gamma_x}(i+1) - y_{\gamma_x}(i))
    \end{align*}
\end{define}
The time parametrization of the path lets us use summation by parts in~\Eqref{eq:summation-by-parts-for-tauberian}, which is a key ingredient in the proof of~\Propref{prop:one-half-of-abelian-theorem}. From the definition, it follows that
\[
    \overline{\nu}_{\e}(x) = \inf_{\gamma_x} \sum_{k = 0}^{\infty} e^{-\e a (i_{k+1}-1)} p \cdot (y_{\gamma_x}(i_{k+1}) - y_{\gamma_x}(i_k)).
\]
It's clear that $\overline{\nu}_{\e}$ and $\nu_{\e}$ must be closely related. In fact,
\begin{prop}
    For all $\e \leq 1$, there is a constant $C$ such that
    \[
        |\overline{\nu}_{\e}(x) - \nu_{\e}(x)| \leq C .
    \]
    \label{prop:time-discretized-stationary-problem}
\end{prop}
\begin{proof}
    Fix a path $\gamma = (v_i)_{i=0}^{\infty}$, and let $y_{\gamma}$ be its time-parametrized version. We have for a mutable constant $C$,
    \begin{align*}
        & \left| \sum_{k = 0}^{\infty} e^{-\e a (i_{k+1} - 1)} p \cdot (y_{\gamma}(i_k+1) - y_{\gamma}(i_k)) - e^{-\e W(v_0,v_k)} p \cdot (v_{k+1} - v_k) \right| \\
        & \qquad \leq \sum_{k = 0}^{\infty} e^{-\e W(v_0,v_k)}(e^{\e (a i_{k+1} - W(v_0,v_k) - a)} - 1) \Norm{p}{\infty} |v_{k+1} - v_k|\\
        & \qquad \leq C_1 \e \sum_{k = 0}^{\infty} e^{-\e a k} \leq C,\\
    \end{align*}
    using $a i_{k+1} \leq W(v_0,v_{k+1}) + a$.
\end{proof}

Next, we will prove the inequality~\Eqref{eq:one-inequality-in-abelian-theorem} in~\Propref{prop:one-half-of-abelian-theorem} for $\e \overline{\nu}_{\e}$ and $t^{-1}\mu(x,t)$. Combining this with~\Propref{prop:time-discretized-stationary-problem} will complete the proof of~\Propref{prop:one-half-of-abelian-theorem}.

Define the $n$ step cost of a time-parametrized path $\gamma$ 
\[
    f_n(\gamma) = \sum_{i=0}^{n-1} p \cdot (y_{\gamma_x}(i+1) - y_{\gamma_x}(i)) =: \sum_{i=0}^{n-1} \lambda_i
\]
and the discounted (in the optimal-control terminology) or Abel sum
\[
    f_{\e}(\gamma) = \sum_{i=0}^{\infty} e^{-\e i} \lambda_i.
\]
Then,
\begin{align}
    \mu(x,n a) & = \inf_{\gamma_x} f_n(\gamma_x), \\
    \overline{\nu}_{\e a^{-1}}(x) & = \inf_{\gamma_x} f_{\e}(\gamma_x).
\end{align}

The following is a summation by parts formula for the discounted sum. When the path $\gamma$ is fixed and understood, we will drop it from the notation in $f_n$ and $f_{\e}$ in the following. For a fixed path $\gamma$, clearly $\lambda_m = f_{m+1} - f_m$. Then for $N \in \Z^+ \cup \{+\infty\}$,
\begin{align}
    \sum_{i=0}^{N} \lambda_i e^{-\e i} 
    & = \sum_{i=0}^{N} (f_{i+1} - f_i) e^{-\e i}, \nonumber\\
    & = e^{-\e (N+1)} f_{N+1} + (e^{\e} - 1) \sum_{i=0}^{N} \frac{f_i}{i} i e^{-\e i} \label{eq:summation-by-parts-for-tauberian} .
\end{align}

Let $C_{\lambda} = \max_i |\lambda_i| < \infty $. For fixed $N$ and small enough $\e$, we have
\begin{align*}
    \sum_{i=0}^{N-1} \frac{|f_i|}{i}  i e^{-\e i} 
    & \leq C_{\lambda} N^2,\\
    \sum_{i=0}^{\infty} i e^{-\e i} 
    & = \frac{1}{\e^2}(1 + O(\e)),\\
    e^{\e} - 1 & = \e + O(\e^2).
\end{align*}
The following proof is taken from~\cite{MR1161156}. 
\begin{proof}[Proof of~\Propref{prop:one-half-of-abelian-theorem}]
    Suppose to the contrary that there exists $\delta$ and $N$ such that $(a n)^{-1} \mu(x,n a) \geq \liminf (\e a^{-1}) \overline{\nu}_{\e a^{-1} }(x) + 3\delta a^{-1}$ for all $n \geq N$. Then, for any $\e_0$, there must exist $\e \leq \e_0$ such that for all $\gamma_x$, we must have $n^{-1} f_n(\gamma_x) \geq \e \overline{\nu}_{\e a^{-1}}(x) + 2\delta$. However (dropping $x$ and $\gamma_x$ from the following display for clarity),
    \begin{align*}
        \e f_{\e} 
        & = \e \sum_{i=0}^{\infty} \lambda_i e^{-\e i} \\
        & = \e(e^{\e} - 1) \sum_{i=0}^{N-1} \frac{f_i}{i} i e^{-\e i} + \e (e^{\e} - 1) \sum_{i=N}^{\infty} \frac{f_i}{i} i e^{-\e i} \\
        & \geq - C N^2 (\e^2 + O(\e^3)) + \e^2 (\e \overline{\nu}_{\e} + 2\delta) \sum_{i=N}^{\infty} i e^{-\e i} - C \e \\
        & \geq \e \overline{\nu}_{\e a^{-1}} + \delta,
    \end{align*}
    when $\e_0$ is small enough. We've also used the fact that $\e \overline{\nu}_{\e a^{-1}}$ is bounded in the above display; see~\Eqref{eq:simple-bounds-on-v-epsilon}. Taking an inf over all $\gamma_x$, this results in a contradiction. 

    Since Prop~\ref{prop:time-discretized-stationary-problem} says that $|\nu_{\e} - \overline{\nu}_{\e}| \leq C$ for all small enough $\e$, this completes the proof.
\end{proof}

As an immediate consequence of~\Propref{prop:one-half-of-abelian-theorem}, we get that for each fixed $x$,
\begin{equation}
    -H(p) \leq \liminf_{\e \to 0} \e \nu_{\e}(x).
    \label{eq:liminf-of-stationary-problem}
\end{equation}

The proof of the DPP for $\nu_{\e}$ in~\Propref{prop:discrete-dpp-stationary-problem} is standard and nearly identical to the proof of the DPP for $\mu(x,t)$ (\Propref{prop:DPP-for-discrete-variational-problem-for-comparison-principle}); hence we will omit it. Two immediate consequences of the DPP are the discrete HJB equation it approximately satisfies in~\Propref{prop:approximate-discrete-HJB-for-stationary-problem}, and the uniform-in-$\e$ derivative bound in Proposition~\ref{prop:lipschitz-estimate-for-stationary-problem}. 

\begin{proof}[Proof of~\Propref{prop:lipschitz-estimate-for-stationary-problem}]
    From the variational definition of $\nu_{\e}$ in~\Defref{def:stationary-dual-problem}, it's easy to see that
  \begin{equation}
      -\frac{|p|_{\infty}}{\e a} \leq \nu_{\e}(x) \leq -\frac{|p|_{\infty}}{\e b} \quad \forall~ x \in \Z^d
      \label{eq:simple-bounds-on-v-epsilon}
  \end{equation}
  where $b$ and $a$ are the upper and lower bounds on $\tau(x,\alpha)$ in~\Eqref{eq:basic-assumption-on-edge-weights}. 

  Using the DPP in~\Eqref{eq:discrete-dpp-stationary-problem}, we get for fixed $\alpha \in A$,
    \[
        \begin{split}
            \nu_{\e}(x)  & \leq p \cdot \alpha + (1 - \e \tau(x,\alpha)) \nu_{\e}(x + \alpha),\\
            -\mathcal{D}\nu_{\e}(x,\alpha) & \leq C,
        \end{split}
    \]
    where we've used the bound in~\Eqref{eq:simple-bounds-on-v-epsilon}, and the boundedness of $\tau(x,\alpha)$. Repeating the argument for $\nu_{\e}(x+\alpha)$ completes the proof. 
\end{proof}


\begin{proof}[Proof of~\Propref{prop:approximate-discrete-HJB-for-stationary-problem}]
 Taylor expand the exponential in the DPP in~\Eqref{eq:discrete-dpp-stationary-problem}, and use the bound on $\nu_{\e}$ in~\Eqref{eq:simple-bounds-on-v-epsilon} to get
 \begin{align*}
    -C\e & \leq\nu_{\e}(x) + \sup_{\alpha \in A} \big( -\alpha \cdot p - (1-\e \tau(x,\alpha)) \nu_{\e}(x+\alpha) \big) \leq C \e ,\\
    -C\e & \leq \sup_{\alpha \in A} \big( -\alpha \cdot p - \mathcal{D}\nu_{\e}(x,\alpha) + \e \tau(x,\alpha) \mathcal{D}\nu_{\e} + \e \tau(x,\alpha) \nu_{\e}(x) \big) \leq C\e .
 \end{align*}
 Using the bounds on $\tau(x,\alpha$ and the Lipschitz estimate on $\nu_{\e}$ in the above, we get
 \begin{align*}
   -C\e & \leq \e \nu_{\e}(x) + \sup_{\alpha \in A} \left( \frac{-\alpha \cdot p - \mathcal{D}\nu_{\e}(x,\alpha)}{\tau(x,\alpha)} \right) \leq C\e.  
 \end{align*}
\end{proof}

Finally, we prove~\Lemref{lem:compactness-argument-for-lower-bound}.
\begin{proof}[Proof of~\Lemref{lem:compactness-argument-for-lower-bound}]
Using the discrete HJB equation for $\nu_{\e}$ from~\Propref{prop:approximate-discrete-HJB-for-stationary-problem}, we get
\[ 
  \e \nu_{\e}(x,\w) + \mathcal{H}(\nu_{\e},p,x,\w) \leq C \e  \quad \forall~ x \in \Z^d. 
\]
Letting $\hat{\nu}_{\e}(x,\w) = \nu_{\e}(x,\w) - \nu_{\e}(0,\w)$, we get 
\[ 
   \e \nu_{\e}(0,\w) + \e \hat{\nu}_{\e}(x,\w) + \mathcal{H}(\hat{\nu}_{\e},p,x,\w) \leq C\e. 
\]
Using the definition of the discrete Hamiltonian, we get for each $\alpha \in A$,
\begin{equation}
  \e \nu_{\e}(0,\w) + \e \hat{\nu}_{\e}(x,\w) + \frac{-p\cdot\alpha-\mathcal{D}\hat{\nu}_{\e}(x,\alpha,\w)}{\tau(x,\alpha,\w)} \leq C\e. 
  \label{eq:step-in-lower-bound-of-variational-formula}
\end{equation}
Since $\hat{\nu}_{\e}$ is normalized to zero at the origin, and inherits the uniform-in-$\e$ Lipschitz estimate on $\nu_e$ in~\Propref{prop:lipschitz-estimate-for-stationary-problem}, we have
\[ 
  C = \sup_{\e} \left\{ \Norm{\hat{\nu}_{\e}(y,\w)(1+|y|)^{-1}}{\infty} + \Norm{\hat{\nu}_{\e}}{Lip} \right\} < \infty .
\]
As a consequence of its definition, $\nu_{\e}(x,\w)$ is stationary. Therefore its derivative is stationary and mean-zero. Hence $\mathcal{D}\hat{\nu}_{\e}$ is stationary and mean-zero, while $\hat{\nu}_{\e}$ itself is not stationary, in general.

%
Let $\psi(\alpha,\w)$ be an $L^2(\W)$ weak limit of $\mathcal{D}\hat{\nu}_{\e}(0,\alpha,\w)$ (as $\e \to 0$) for each $\alpha \in \{e_1,\ldots,e_d\}$. It's easy to check that $\psi(V^x\w)$ is the weak limit of $\mathcal{D}\hat{\nu}_{\e}(V^x \w)$ for all $x \in \Z^d$. Hence, with a slight abuse of notation, we use the translation group to define $\psi(x,\alpha,\w) = \psi(\alpha,V^x\w)$. In fact, $\psi(\alpha,x,\w)$ comes from discrete differentiating a function; we show this next.

For any $\e$ and any fixed loop $\gamma_{xx}$, $\mathcal{D}\hat{\nu_{\e}}$ sums to zero over the loop. Since this is a linear condition, it is preserved under the weak-limit, and hence $\psi(x,\alpha,\w)$ also sums to zero almost surely over the loop. Since there are only a countable number of loops and a countable number of points, $\psi$ sums over all loops at every location to zero almost surely. Hence, there is a function $u(x,\w)$ such that $\mathcal{D}u(x,\alpha,\w) = \psi(x,\alpha,\w)$. 

First, pass to a subsequence such that $\mathcal{D}\hat{\nu}_{\e} \to \mathcal{D}u$ weakly in $L^2$. We know from~\Eqref{eq:liminf-of-stationary-problem} that $\liminf \e \nu_{\e}(0) \geq -H(p)$ almost surely along this sequence. Taking a liminf as $\e \to 0$ in~\Eqref{eq:step-in-lower-bound-of-variational-formula}, we have for each fixed $x \AND \alpha$ and any nonnegative bounded function $g(\w) \in L^2(\W)$,
\[ 
  \int g(\w) \left(\frac{ - p\cdot\alpha-\mathcal{D}u(x,\alpha,\w)}{\tau(x,\alpha,\w)} \right) \Prob(d\w) \leq \Heff(p) \int g(\w) \Prob(d\w). 
\]
This tells us that the inequality must hold almost surely. We can take a supremum over $\alpha \in A$ and then over $x \in \Z^d$ to get
\[ 
  \sup_x \mathcal{H}(u,p,x,\w) \leq \Heff(p) \quad \almostsurely.
\]
\end{proof}

\section{Proof of the Algorithm}
\label{sec:explicit-algorithm-to-produce-a-minimizer}
We first complete the proof of the comparison principle for
supersolutions
in Proposition~\ref{prop:other-inequality-comparison-principle}.

\begin{proof}[Proof of~\Propref{prop:other-inequality-comparison-principle}]
  Let $C = \inf_{x \in \Z^d} \mathcal{H}(\phi,x) > -\infty$ without
  loss of generality. Let
  \begin{equation*}
    \zeta(x,t) = \phi(x) - t C.
  \end{equation*}
  Again the proof is by induction on $n$. For any $n > 1$, assume for
  some finite constant $M$,
  \[
    \zeta(x,t) \geq \mu(x,t) - M \quad \forall  x \in Z^d ,~ t <
    (n-1)a.
  \]


  Following~\Eqref{eq:comparison-principle-sub-super-solution-display},
  we get
  \[
    \zeta(x,t) \geq \inf_{\alpha \in A} \{ \zeta(x,t - \tau(x,\alpha))
    + \lambda(x,\alpha) \}.
  \]
  Again, there are two cases: Suppose first that
  $t \geq \min_c \tau(x,c)$; we may use the induction hypothesis to
  get
  \[
    \begin{split}
      \zeta(x,t) \geq \inf_{c \in A} \left\{ \mu(x+c,t-\tau(x,c)) +
        \lambda(x,c) \right\} \land \phi(x) - M \quad \forall  x \in
      \Z^d .
    \end{split}
  \]
  When $t < \min_c \tau(x,c) \leq b$, we have
  $\zeta(x,t) = \phi(x) - t C \geq \mu(x,t) - \max(b C, 0) $. This
  completes the induction step. The $n=1$ case is the same as the
  previously considered case and gives $M = \max(bC,0)$.
\end{proof}

Using~\Propref{prop:other-inequality-comparison-principle} and
following the argument in the upper bound
in~Lem\-ma~\ref{lem:upper-bound-for-variational-formula-using-comparison-principle}
gives~\Corref{cor:inf-sup-inequality-for-limiting-Hamiltonian}

Let $A^+ = \{e_1, \ldots, e_d\}$. Let
$\tilde{\tau}\colon A^+ \times \W \to \R$ be a function representing
the edge weight at the origin. For example, it could consist of $d$
i.i.d.~edge weights, one for each direction. Let
$x=(x_1,\ldots,x_d) \in \Z^d$. Under the symmetry
assumption~\Eqref{eq:symmetry-assumption-on-medium}, the edge weight
function is given by
\[
  \tau(x,\alpha,\w) = \tilde{\tau}\left(\alpha,V_{e_1}^{x_1}\cdots
    V_{e_d}^{x_d}\w \right) = \tilde{\tau}\bigl(\alpha,V^{\sum_{i=1}^d
      x_i}\w \bigr).
\]
This means that $\tau(\,\cdot\,, \cdot \,,\w)$ is constant along the
hyperplanes $\{ x \in \Z^d : \sum_{i=1}^d x_i = z \}$ for each
$z \in \Z \AND \w \in \W$. Despite this, the time constant is not
\emph{that} obvious, although one ought to be able to calculate it.

\begin{prop}  \label{prop:simplifying-the-set-of-functions-S}
  If $\phi \in S$ and the symmetry
  condition~\Eqref{eq:symmetry-assumption-on-medium} holds, the
  derivative points in the $\sum_i e_i$-direction; i.e.,
  \[
    \mathcal{D}_{\alpha}\phi(x,\w) = \mathcal{D}_{e_1}\phi(x,\w) \quad
    \forall \alpha \in A \ \text{and}\ \forall x \in \Z^d ~\almostsurely
  \]
\end{prop}
\begin{proof}
  The derivative of $\phi$ sums to $0$ over any discrete loop in
  $\Z^d$: for any $i \neq j \in \{1,\ldots,d\}$
  \begin{equation}\begin{split}
    \mathcal{D}_{e_i}\phi(x,\w) + \mathcal{D}_{e_j}\phi(x+e_i,\w)  + \mathcal{D}_{-e_i}\phi(x + e_j + e_i,\w)&\\*
{}     +
    \mathcal{D}_{-e_j}\phi(x+e_j,\w) &= 0.
  \end{split}
\end{equation}
  Since the derivative is stationary,
  $\mathcal{D}_{\alpha}\phi(x,\w) =
  -\mathcal{D}_{-\alpha}\phi(x+\alpha,\w)$ and $V^{e_i} = V^{e_j}$, we have
  \[
    \begin{split}
      \mathcal{D}_{e_i}\phi(x,\w) - \mathcal{D}_{e_j}\phi(x,\w)
      & = \mathcal{D}_{e_i}\phi(x,V\w) - \mathcal{D}_{e_j}\phi(x,V\w).
    \end{split}
  \]
  This means that
  $\mathcal{D}_{e_i}\phi(x,\w) - \mathcal{D}_{e_j}\phi(x,\w)$ is
  invariant under $V$, and hence it must be a constant almost
  surely. Since it also has zero mean, it follows that
  \[
    \mathcal{D}_{\alpha}\phi(x,\w) = \mathcal{D}_{e_1}\phi(x,\w) \quad
    \forall \alpha \in A^+ \ \almostsurely
  \qedhere\]
\end{proof}

Recall the discrete Hamiltonian under the symmetry assumption,
\begin{equation*}
  \mathcal{H}_{\rm sym}(t,p,\w) := \sup_{\alpha \in A^+} \frac{|t + p\cdot\alpha|}{\tau(0,\alpha,\w)} . 
\end{equation*}

\begin{proof}[Proof of~\Propref{prop:variational-formula-in-symmetric-situation}]
  For each $\phi \in S$, let $f(\w) = -
  \mathcal{D}\phi(0,e_1,\w)$. For each $\alpha \in A^+,\,x \in \Z^d$,
  we get
  \begin{multline}
    \max\left\{ \frac{-p\cdot\alpha + \phi(x+\alpha) - \phi(x)}{\tau(x,\alpha)}, \frac{-p\cdot(-\alpha) + \phi(x) - \phi(x+\alpha)}{\tau(x+\alpha,-\alpha)} \right\} 
    = \\*\frac{\left|-p\cdot\alpha +
      \mathcal{D}_{\alpha}\phi(x)\right|}{\tau(x,\alpha)},
    \label{eq:algorithm-hsym-form-when-tau-undirected}
  \end{multline}
  by using~\Eqref{eq:algorithm-undirected-edge-weight-assumption}, the undirectedness of the edge weights. Both
  terms on the \LHS\ appear when taking a max over $x$ and
  $\alpha$ in~\Thmref{thm:discrete-variational-formula-for-H-bar}. Hence,
  using~\Propref{prop:simplifying-the-set-of-functions-S}, we get
  \[
    \max_x \mathcal{H}(\phi,p,x) = \max_x
    \mathcal{H}_{\rm sym}(f(\w),p,\w) \quad\almostsurely
  \qedhere\]
\end{proof}

In the following, we will write $\mathcal{H}_{\rm sym}(f,\w)$ instead of
$\mathcal{H}_{\rm sym}(f,p,\w)$ since $p$ plays no role in the
argument. The idea behind the algorithm is simple. At each iteration, it tries to reduce
the essential supremum over $\w$ by modifying $f(\w)$, while ensuring
that the modified $f$ remains inside $F$. It turns out that if the
algorithm fails to reduce the sup of $f$, then $f$ must be a
minimizer. We explain what we're trying to do in each step in the
proof of convergence of the algorithm. So we suggest skimming the
algorithm first, and returning to the definition of each step when
reading the proof.\newline

\subsection*{Start algorithm}
\begin{enumerate}
\item \label{step:find-distance-between-sup-and-mean} Start with any
  $f_0 \in F$. Let $\mu_0 = \E[\mathcal{H}_{\rm sym}(f_0,\w)]$, and let
  \[
    d_0 = \esssup_{\w \in \W} \mathcal{H}_{\rm sym}(f_0,\w) - \mu_0.
  \]
  If $d_0 = 0$, stop.
\item \label{step:define-min-set-S-and-I} Define the sets
  \begin{align}
    \textup{MIN}_0 & :=  \{ \w : \mathcal{H}_{\rm sym}(f_0,\w) = \min_x \mathcal{H}_{\rm sym}(x,\w)  \},\\
    S & :=  \{ \w : \mathcal{H}_{\rm sym}(f_0,\w) > \mu_0 \},\vphantom{\min_x}\\
    I & :=  \{ \w : \mathcal{H}_{\rm sym}(f_0,\w) < \mu_0 \}.
        \label{eq:definition-of-the-sets-MIN-S-I}
  \end{align}
  If
  \[ \esssup_{\w \in \textup{MIN}_0} \mathcal{H}_{\rm sym}(f_0,\w) = \esssup_{\w
      \in \W} \mathcal{H}_{\rm sym}(f_0,\w), \] stop.
\item \label{step:construct-delta-f} Let $\Delta f^*(\w)$ be such that
  \[ \mathcal{H}_{\rm sym}(f_0 + \Delta f^*(\w),\w) =
    \mathcal{H}_{\rm sym}(x^*(\w),\w),\]
  where $x^*(\w) = \argmin_{x} \mathcal{H}_{\rm sym}(x,\w)$.  Define the
  sets
  \begin{align*}
    S_+ & :=  \{ \w \in S \setminus \textup{MIN}_0 : f_0 > x^*(\w) \},\\
    S_- & :=  \{ \w \in S \setminus \textup{MIN}_0 : f_0 < x^*(\w)  \}.
  \end{align*}
  Let
  \begin{equation}
    \Delta f(\w) = \begin{cases} 
                             \max(- a (\mathcal{H}_{\rm sym}(f_0,\w) - \mu_0),~\Delta f^*(\w) ), 
                             & \w \in S_+, \\
                             \min( a (\mathcal{H}_{\rm sym}(f_0,\w) - \mu_0),~\Delta f^*(\w) ), 
                             & \w \in S_- , \\
                             a \xi (\mu_0 - \mathcal{H}_{\rm sym}(f_0,\w) ), & \w \in I ,\\
                             0,   & \text{elsewhere},
                                          \end{cases}
                         \label{eq:algorithm-def-of-delta-f}
                       \end{equation}
                       where
                       \begin{equation}
                         \xi = - \frac{ \E [ \Delta f(\w), S_+ \cup S_-]}{ \E[ a (\mu_0 - \mathcal{H}_{\rm sym}(f_0,\w) ), I]  }.
                         \label{eq:algorithm-xi-definition}
                       \end{equation}
                       Let $f_1 = f_0 + \Delta f(\w)$. Return to step
                       $1$.
                     \end{enumerate}\vspace{-7pt}
\subsection*{End algorithm}

                     Recall~\Thmref{thm:algorithm-convergence}:

\begin{theorem*}
  Let $\{f_n\}_{n=0}^{\infty}$ be the sequence obtained by iterating
  the algorithm on an initial point $f_0 \in F$. Let $d_n = \esssup_{\W} \mathcal{H}_{\rm sym}(f_n(\w),p,\w) - \E[ \mathcal{H}_{\rm sym}]$. There are three 
  possibilities:
  \begin{enumerate}
  \item If it terminates in a finite number of steps with
    $d_n = 0$, we have a minimizer that's a corrector.
  \item If it terminates in a finite number of steps with $d_n > 0$, we
    have a minimizer that's not a corrector.
  \item If it does not terminate, we produce a corrector in the
    limit. That is, $f_n \to f_{\infty}$ in measure and
    $\mathcal{H}_{\rm sym}(f_{\infty},p,\w) = \Heff(p)~\almostsurely$
  \end{enumerate}
\end{theorem*}

Next, we briefly illustrate two different types of minimizers obtained
by the algorithm. Suppose we're in two dimensions. If $p=(1,1)$, then
\[
  f_0(\w) = \frac{\min_{i=1,2} \tau(0,e_i)}{\E[\min_{i=1,2}
    \tau(0,e_i)]} - 1
\]
is a corrector. If $p=(-1,1)$, then $\mathcal{H}_{\rm sym}(\,\cdot\,,p,\w)$
takes its minimum value at
\[
  \frac{(\tau(0,e_2) - \tau(0,e_1))}{(\tau(0,e_1)+\tau(0,e_2))} :=
  f_0.
\]
If we assume that $\tau(0,e_1)$ and $\tau(0,e_2)$ are i.i.d.,
$\E[f_0] = 0$. Then we're clearly at a minimizer, but
$f_0$ is not a corrector.

We need the following lemma to
prove~\Thmref{thm:algorithm-convergence}.
\begin{lemma}\label{lem:H-sym-is-convex-and-its-min-is-measurable}
  The function $\mathcal{H}_{\rm sym}(x,\w)$ has the following properties:
  \begin{enumerate}
  \item For each $\w$, it is convex in $x$.
  \item It has a unique measurable minimum $x^*(\w)$.
  \item Its left and right derivatives (in $x$) satisfy
    \begin{alignat*}{2}
      D_{-} \mathcal{H}_{\rm sym}(x,\w) &\in [b^{-1},a^{-1}], \quad & x &\geq x^*(\w),\\*
      D_+ \mathcal{H}_{\rm sym}(x,\w) &\in [-a^{-1},-b^{-1}], \quad & x &\leq x^*(\w).
    \end{alignat*}
  \end{enumerate}
\end{lemma}
We will prove~\Lemref{lem:H-sym-is-convex-and-its-min-is-measurable}
after proving~\Thmref{thm:algorithm-convergence}.
\begin{figure}
  \centering
  \includegraphics[height=0.3\textheight]{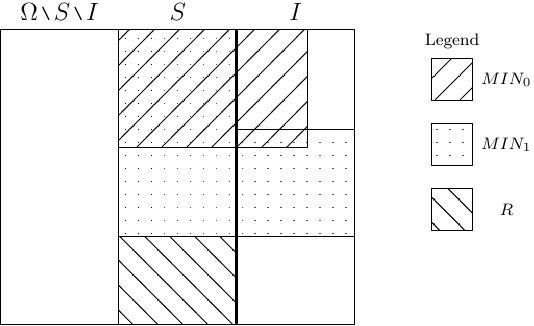}
  \caption{Sketch of sets in algorithm. The outer rectangle represents
    the probability space $\W$. It shows the situation where
    $R \neq S \setminus \textup{MIN}_0$. In this case, there is a possibility
    that the algorithm may terminate in the next step.}
  \label{fig:sets-in-algorithm}
\end{figure}

\pagebreak[2]
\begin{proof}[Proof of~\Thmref{thm:algorithm-convergence}]\hfill
\begin{enumerate}
\item  In the first step, we compute $d_0$, the distance
  between the mean and supremum of $\mathcal{H}_{\rm sym}(f,\w)$. If
  $d_0 = 0$, $f$ must be a corrector, and
  by~\Corref{cor:inf-sup-inequality-for-limiting-Hamiltonian} it must
  be a minimizer. Therefore, we stop the algorithm.
\item  $\textup{MIN}_0$ is the set on which $\mathcal{H}_{\rm sym}(f_0,\w)$
  cannot be lowered further. $S$ and $I$ are the sets on which
  $\mathcal{H}_{\rm sym}(f_0,\w)$ is higher and lower than its mean
  $\mu_0$, respectively. $f_0$~will be modified on these two sets in step 3.

Lemma~\ref{lem:H-sym-is-convex-and-its-min-is-measurable} says that
  $\mathcal{H}_{\rm sym}(\,\cdot\,,\w)$ is convex and has a minimum at
  $x^*(\w)$. So there is the possibility of the algorithm getting
  ``stuck'' at a minimum of $\mathcal{H}_{\rm sym}$. That is, $f_0$ might
  be such that
  $\mathcal{H}_{\rm sym}(f_0,\w) = \mathcal{H}_{\rm sym}(x^*(\w),\w)$ on a set
  of positive measure. Suppose we also have
  \[
    \esssup_{\w \in \textup{MIN}_0} \mathcal{H}_{\rm sym}(f_0,\w) = \esssup_{\w \in
      \W} \mathcal{H}_{\rm sym}(f_0,\w).
  \]
  Then, for any other $g \in F$, we clearly have
  $\mathcal{H}_{\rm sym}(g(\w),\w) \geq \mathcal{H}_{\rm sym}(f_0(\w),\w)$ on
  $\textup{MIN}_0$. Hence $f_0$ must be a minimizer, and we stop the algorithm.

\item 
 $\Delta f$ is first defined on the sets $S_+$ and
  $S_-$ so that the supremum falls. Then, $\Delta f$ is defined on $I$
  so that it satisfies
  \[
   \E[\Delta f] = 0,
  \]
  and therefore $f_0 + \Delta f$ remains in the set $F$.
    
  We must 
ensure that $\xi$
  in~\Eqref{eq:algorithm-xi-definition} is not infinite. Note that
  $\E[(\mu_0 - \mathcal{H}_{\rm sym}(f_0,\w) ), I] > 0$ since we've assumed that $\mathcal{H}_{\rm sym}(f_0,\w) < \mu_0$ on $I$. Hence $\Delta f$ is well-defined on $I$. Since 
  \begin{equation*}
    \E[ \left( \mathcal{H}_{\rm sym}(f_0,\w) - \mu_0 \right) , S ]  
    = 
    \E[ \left( \mu_0 - \mathcal{H}_{\rm sym}(f_0,\w) \right) , I ] ,
  \end{equation*}
  we have
  \[
    \begin{split}
      | \E[\Delta f, S_+ \cup S_-] |
      & \leq a \E[(\mathcal{H}_{\rm sym}(f_0,\w) - \mu_0),S_+ \cup S_-],  \\
      & \leq a \E[\mu_0 - \mathcal{H}_{\rm sym}(f_0,\w), I] .
    \end{split}
  \]
  Therefore from~\Eqref{eq:algorithm-xi-definition},
  \begin{equation}
    -1 \leq \xi  \leq 1  .
    \label{eq:algorithm-xi-is-bounded}
  \end{equation}
\end{enumerate}

  Next, we claim that if $\esssup \mathcal{H}_{\rm sym}(f_1)$ does not
  fall sufficiently at the end of step $3$, it will terminate in the next
  iteration.
  \begin{claim}
    If
    \begin{equation}
      \esssup_{\w \in \W} \mathcal{H}_{\rm sym}(f_1,\w) > \esssup_{\w \in \W} \mathcal{H}_{\rm sym}(f_0,\w) - d_0 \frac{a}{b}, 
      \label{eq:suppose-that-H-f1-does-not-fall-enough}
    \end{equation}
    the algorithm will terminate when it goes to step $2$ in the
    following iteration. That is, we will have
    \[
      \esssup_{\w \in \W} \mathcal{H}_{\rm sym}(f_1,\w) =
      \esssup_{\textup{MIN}_1}\mathcal{H}_{\rm sym}(f_1,\w),
    \]
    where $\textup{MIN}_1$ is defined
    in~\Eqref{eq:definition-of-the-sets-MIN-S-I} with $f_0$ replaced
    by $f_1$.
    \label{claim:intermediate-termination-possibility-if-sup-does-not-fall-enough}
  \end{claim}

  We will show
  Claim~\ref{claim:intermediate-termination-possibility-if-sup-does-not-fall-enough}
  after completing the proof of the theorem. Next, we prove that if
  the algorithm does not terminate in either
  step~\ref{step:find-distance-between-sup-and-mean}
  or~\ref{step:define-min-set-S-and-I}, we produce a minimizing
  sequence. Suppose that the algorithm does not terminate, and let
  $f_n$ be the $n$\textsuperscript{th} iterate. Let $d_n$ be the
  corresponding distance between the ess\,sup and mean of the
  Hamiltonian. Claim~\ref{claim:intermediate-termination-possibility-if-sup-does-not-fall-enough}
  gives us the estimate
  \[
    \esssup_{\w \in \W} \mathcal{H}_{\rm sym}(f_n,\w) \leq \esssup_{\w \in
      \W} \mathcal{H}_{\rm sym}(f_{n-1},\w) - d_{n-1} \frac{a}{b}.
  \]
  Since $\mathcal{H}_{\rm sym} \geq 0$, we must have
  $\sum_{n=0}^{\infty} d_n < \infty$. The form of $\Delta f_n$
  in~\Eqref{eq:algorithm-def-of-delta-f} implies
  \begin{gather*}
    |\Delta f_n(\w)| \leq d_n,  \quad \w \in S_+ \cup S_-,\\*
    \E[ |\Delta f_n(\w)|, I ]  = \left| \E[ \Delta f_n, S_+ \cup S_- ] \right|
                                \leq d_n.
  \end{gather*}
  Hence, $\sum_{i=0}^{\infty} \Delta f_n$ is absolutely summable in
  $L^1(\W)$ and therefore $f_n \to f_{\infty}$ in measure and in
  $L^1$. Since $\mathcal{H}_{\rm sym}(f_n(\w),\w)$ is uniformly bounded, so is
  the sequence $\{f_n\}_{n=0}^{\infty}$.

  Then, for any $p > 1$,
  \[
    \begin{split}
      0 &=  \lim_n d_n \\
      &= \lim_{n \to \infty} \left( \esssup \mathcal{H}_{\rm sym}(f_n(\w),\w) - \int \mathcal{H}_{\rm sym}(f_n(\w),\w)  \Prob(d\w) \right) ,\\
      &\geq \Norm{\mathcal{H}_{\rm sym}(f_{\infty},\w)}{p} - \int \mathcal{H}_{\rm sym}(f_{\infty},\w) \geq 0,\\
    \end{split}
  \]
  using the continuity of the $\mathcal{H}_{\rm sym}(x,\w)$ in the $x$-variable,
  its nonnegativity, and the bounded convergence theorem. This implies that $\mathcal{H}_{\rm sym}(f_{\infty},\w)$ is a constant and thus, $f_{\infty}$ is a corrector. This completes the proof 
  except
  for~\Claimref{claim:intermediate-termination-possibility-if-sup-does-not-fall-enough}. 

\end{proof}

\begin{proof}[Proof of~\Claimref{claim:intermediate-termination-possibility-if-sup-does-not-fall-enough}]
  It will be useful to refer to~\Figref{fig:sets-in-algorithm}, which
  is a visual representation of the sets defined in the algorithm.
  Let
  \[
    R := \{ \w \in S \setminus \textup{MIN}_0 : a|\mathcal{H}_{\rm sym}(f_0,\w) -
    \mu_0)| < \Delta f^*(\w) \}
  \]
  be the set on which we can modify $f_0$ without hitting the minimum
  of $\mathcal{H}_{\rm sym}(f_0,\w)$; i.e.,
  $\mathcal{H}_{\rm sym}(f_1,\w) > \mathcal{H}_{\rm sym}(x^*(\w),\w)$. From the
  definition of $\Delta f$ and the bound on the derivatives of
  $\mathcal{H}_{\rm sym}(\,\cdot\,,\w)$
  in~\Lemref{lem:H-sym-is-convex-and-its-min-is-measurable}, we have
  \begin{align*}
       \mathcal{H}_{\rm sym}(f_0,\w)  - \frac{1}{a}a( & \mathcal{H}_{\rm sym}(f_0,\w) - \mu_0)\\*
      \leq \mathcal{H}_{\rm sym}(f_1,\w)\\*
          \leq  \mathcal{H}_{\rm sym}(f_0,\w) - \frac{1}{b} a (\mathcal{H}_{\rm sym}(f_0,\w) - \mu_0) \quad \w \in R 
      ~\almostsurely.
 \end{align*}
  Therefore,
  \begin{equation}
    \mu_0 \leq  \esssup_{\w \in R} \mathcal{H}_{\rm sym}(f_1,\w) \leq  \esssup_{\w \in R}  \mathcal{H}_{\rm sym}(f_0,\w) - \frac{da}{b}.
    \label{eq:bound-on-H-f1-on-set-R}
  \end{equation}
  Similarly for $\w \in I$, we use the bound on $\xi$
  in~\Eqref{eq:algorithm-xi-is-bounded} and the derivative bound
  in~\Lemref{lem:H-sym-is-convex-and-its-min-is-measurable} to get
  \begin{equation}
    \mathcal{H}_{\rm sym}(f_1,\w) \leq \mathcal{H}_{\rm sym}(f_0,\w) + \left| \xi \right| (\mu_0 - \mathcal{H}_{\rm sym}(f_0,\w)) \leq \mu_0, \quad \w \in I ~\almostsurely 
    \label{eq:bound-on-H-f1-on-set-I}
  \end{equation}
  It's clear that $S \setminus R \subset \textup{MIN}_1$, since $S \setminus R$
  represents the set on which we will hit the minimum of
  $\mathcal{H}_{\rm sym}$ when modifying $f_0$. Since $\Delta f = 0$ on
  $S \cap \textup{MIN}_0$, we must also have $S \cap \textup{MIN}_0 \subset \textup{MIN}_1$.

  Consider the claim
  in~\Eqref{eq:suppose-that-H-f1-does-not-fall-enough} again; this
  and equations~\Eqref{eq:bound-on-H-f1-on-set-R}
  and~\Eqref{eq:bound-on-H-f1-on-set-I} imply that we can ignore the
  sets $R$ and $I$ when taking a sup over $\W$. It's clear that we can
  ignore $\W \setminus S \setminus I$ too, since
  $\mathcal{H}_{\rm sym}(f_0,\w) = \mu_0$ on this set. Summarizing, we get
  \[
    \begin{split}
      \esssup_{\W} \mathcal{H}_{\rm sym}(f_1,\w) & = \esssup_{S \setminus
        R ~\cup~ S \cap \textup{MIN}_0 } \mathcal{H}_{\rm sym}(f_1,\w) =
      \esssup_{\textup{MIN}_1} \mathcal{H}_{\rm sym}(f_1,\w). 
    \end{split}
  \]
  Hence, the algorithm terminates at step $2$ in the next iteration.
\end{proof}

%
To finish, we complete the proof
of~\Lemref{lem:H-sym-is-convex-and-its-min-is-measurable}.
\begin{proof}[Proof
  of~\Lemref{lem:H-sym-is-convex-and-its-min-is-measurable}]
  Clearly $\mathcal{H}_{\rm sym}$ is convex. Its minimum is unique since
  $\tau(0,\alpha,\w) \leq b$, and therefore, it cannot have a ``flat
  spot'' parallel to the $t$-axis.

  $\mathcal{H}_{\rm sym}$ can only take its minimum at a minimum of $|t +
  p\cdot\alpha|/\tau(0,\alpha,\w)$ for some $\alpha \in A^+$, or when $t$ is
  such that $|t + p\cdot \alpha_1|/\tau(0,\alpha_1,\w) = |t + p\cdot
  \alpha_2|/\tau(0,\alpha_2,\w)$ for any $\alpha_1,\alpha_2 \in A^+$. There are
  only a finite number of such possibilities; we can compute all of them
  explicitly, and hence $x^*(\w)$ is measurable. 

  We have $D_{-} \mathcal{H}_{\rm sym}(t,\w) \in [b^{-1},a^{-1}]$ or $D_+
  \mathcal{H}_{\rm sym}(t,\w) \in [-a^{-1},-b^{-1}]$ for all $t$, since
  $\mathcal{H}_{\rm sym}$ is a max of linear functions each with absolute
  slopes between $b^{-1}$ and $a^{-1}$.
\end{proof}

Consider the following special case of \FPP~under the symmetry
assumption. Suppose we have a periodic medium with equal periods in all
directions; i.e., the translations satisfy for some fixed $n$
 \[
     \tau(0,\cdot,V^n \w) = \tau(0,\cdot,V \w) ~\almostsurely.
 \]
Then, the edge weights $\vec{q}(\w) = (\tilde{\tau}(e_1,\w), \ldots, \tilde{\tau}(e_d,\w))$ can only take a finite number of different values $\{\vec{q}_0,\ldots,\vec{q}_{n-1}\}$. 
Define the sets
\[
  A_i := \{ \w \in \W : \tau(0,\cdot,\w) = \vec{q}_i \}, 
  \quad i=0,\ldots,n-1. 
\]
 Periodicity forces the constraint $\Prob(A_i) = 1/n$. The set $F$ of functions in~\Eqref{eq:set-S-under-symmetry-assumption} can be restricted to
 \begin{equation}
     F := \left\{ f(\w) : f(\w) = \sum_{i=0}^{n-1} f_i 1_{A_i}(\w),~f_i \in \R,~E[f] = 0, \, \Norm{f}{\infty} < \infty \right\},
   \label{eq:set-of-functions-F-in-variational-formula}
 \end{equation}
 and the algorithm continues to produce a minimizer. 
 
  Periodic homogenization has been well studied and there are many algorithms to produce the effective Hamiltonian; see for example,~\citet{gomes_computing_2004} or~\citet{oberman_homogenization_2009}. Our contribution here is that the algorithm works even if $\tau(0,\cdot,\w)$ takes an uncountable number of values; i.e., the period is infinite. It's worth stating that our algorithm appears to be quite fast, even compared to large time methods like~\citet{oberman_homogenization_2009}.

 \begin{remark}
   The symmetry assumption is a massive simplification, and removing
   this is a real challenge. If the generating translations $V^{e_i}$ are rationally
   related, we ought to be able to generalize the algorithm with a
   little work. However, taking this route in general---solving the loop/cocycle
   condition as
   in~\Propref{prop:simplifying-the-set-of-functions-S}--- is probably
   hopeless. 
 \end{remark}

\ack I'd like to acknowledge S.~Chatterjee, R.V.~Kohn and S.~R.~S.\ 
Varadhan for their advice and encouragement. Special thanks to S~.R~.S\ 
Varadhan for numerous helpful suggestions. I'd also like to thank
N.~Georgiou, F.~Rassoul-Agha and T.~Sepp\"al\"ainen for helping me
understand their work. Special thanks to T.~Sepp\"al\"ainen for
catching a serious mistake in an earlier version of this paper. I also
appreciate the comments of an anonymous reviewer, and especially his or
her encouragement to rewrite the paper the way it was meant to be
written. I had several helpful discussions with J.\,Portegis, M.\,Harel
and B.\,Mehrdad. I was partially supported by the NSF grants
DMS-1007524 and DMS-1208334.

\bibliographystyle{chicago}
\bibliography{/home/arjun/master-bibtex}

\end{document}